\documentclass[a4paper,10pt]{article}

\usepackage[T1]{fontenc} 
\usepackage[utf8]{inputenc} 
\usepackage[english]{babel}
\usepackage[left=1.5in,textwidth=30pc,textheight=54.5pc]{geometry} 

\usepackage{amsmath,amsfonts,amssymb,amsthm,mathrsfs,mathtools}
\usepackage[hidelinks,breaklinks]{hyperref}
\usepackage{paralist}
\usepackage{enumitem}
\usepackage{textcase} 
\usepackage{filemod} 
\usepackage[longnamesfirst,authoryear]{natbib}
\usepackage{titlesec} 
\usepackage{etoolbox} 
\usepackage{color}
\usepackage{booktabs}
\usepackage{pgfplots}
\pgfplotsset{compat=newest}
\usetikzlibrary{calc}

\makeatletter

\def\cite{\citet}

\numberwithin{equation}{section}
\allowdisplaybreaks

\def\@noindentfalse{\global\let\if@noindent\iffalse}
\def\@noindenttrue {\global\let\if@noindent\iftrue}
\def\@aftertheorem{%
  \@noindenttrue
  \everypar{%
    \if@noindent%
      \@noindentfalse\clubpenalty\@M\setbox\z@\lastbox%
    \else%
      \clubpenalty \@clubpenalty\everypar{}%
    \fi}}

\theoremstyle{plain}
\newtheorem{theorem}{Theorem}[section]
\AfterEndEnvironment{theorem}{\@aftertheorem}

\AfterEndEnvironment{definition}{\@aftertheorem}
\newtheorem{lemma}[theorem]{Lemma}
\AfterEndEnvironment{lemma}{\@aftertheorem}

\AfterEndEnvironment{corollary}{\@aftertheorem}

\theoremstyle{definition}
\newtheorem{remark}[theorem]{Remark}
\AfterEndEnvironment{remark}{\@aftertheorem}

\AfterEndEnvironment{example}{\@aftertheorem}
    
\AfterEndEnvironment{proof}{\@aftertheorem}

\titleformat{name=\section}
  {\center\small\sc}{\thesection\kern1em}{0pt}{\MakeTextUppercase}
\titleformat{name=\section, numberless}
  {\center\small\sc}{}{0pt}{\MakeTextUppercase}

\titleformat{name=\subsection}
  {\bf\mathversion{bold}}{\thesubsection\kern1em}{0pt}{}
\titleformat{name=\subsection, numberless}
  {\bf\mathversion{bold}}{}{0pt}{}

\titleformat{name=\subsubsection}
  {\normalsize\itshape}{\thesubsubsection\kern1em}{0pt}{}
\titleformat{name=\subsubsection, numberless}
  {\normalsize\itshape}{}{0pt}{}

\def\be#1{\begin{equation*}#1\end{equation*}}
\def\ben#1{\begin{equation}#1\end{equation}}
\def\bes#1{\begin{equation*}\begin{split}#1\end{split}\end{equation*}}
\def\besn#1{\begin{equation}\begin{split}#1\end{split}\end{equation}}

\def\ba#1{\begin{align*}#1\end{align*}}

\def\bg#1{\begin{gather*}#1\end{gather*}}
\def\bgn#1{\begin{gather}#1\end{gather}}
\def\bm#1{\begin{multline*}#1\end{multline*}}
\def\bmn#1{\begin{multline}#1\end{multline}}

\usepackage[linewidth=1pt]{mdframed}

\usepackage[textwidth=1.64in]{todonotes}
\let\@@todo\todo
\def\todo#1{\@@todo[color=red,backgroundcolor=red!10,size=\tiny]{#1}}

\usepackage{delimset}
\def\given{\mskip 0.5mu plus 0.25mu\vert\mskip 0.5mu plus 0.15mu}
\newcounter{bracketlevel}%
\def\@bracketfactory#1#2#3#4#5#6{%
\expandafter\def\csname#1\endcsname##1{%
\global\advance\c@bracketlevel 1\relax%
\global\expandafter\let\csname @middummy\alph{bracketlevel}\endcsname\given%
\global\def\given{\mskip#5\csname#4\endcsname\vert\mskip#6}\csname#4l\endcsname#2##1\csname#4r\endcsname#3%
\global\expandafter\let\expandafter\given\csname @middummy\alph{bracketlevel}\endcsname%
\global\advance\c@bracketlevel -1\relax%
}%
}
\def\bracketfactory#1#2#3{%
\@bracketfactory{#1}{#2}{#3}{relax}{0.5mu plus 0.25mu}{0.5mu plus 0.15mu}
\@bracketfactory{b#1}{#2}{#3}{big}{1mu plus 0.25mu minus 0.25mu}{0.6mu plus 0.15mu minus 0.15mu}
\@bracketfactory{bb#1}{#2}{#3}{Big}{2.4mu plus 0.8mu minus 0.8mu}{1.8mu plus 0.6mu minus 0.6mu}
\@bracketfactory{bbb#1}{#2}{#3}{bigg}{3.2mu plus 1mu minus 1mu}{2.4mu plus 0.75mu minus 0.75mu}
\@bracketfactory{bbbb#1}{#2}{#3}{Bigg}{4mu plus 1mu minus 1mu}{3mu plus 0.75mu minus 0.75mu}
}
\bracketfactory{clc}{\lbrace}{\rbrace}
\bracketfactory{clr}{(}{)}
\bracketfactory{cls}{[}{]}
\bracketfactory{abs}{\lvert}{\rvert}
\bracketfactory{norm}{\Vert}{\Vert}
\bracketfactory{floor}{\lfloor}{\rfloor}
\bracketfactory{ceil}{\lceil}{\rceil}
\bracketfactory{angle}{\langle}{\rangle}
\bracketfactory{ip}{\langle}{\rangle}

\let\original@left\left
\let\original@right\right
\renewcommand{\left}{\mathopen{}\mathclose\bgroup\original@left}
\renewcommand{\right}{\aftergroup\egroup\original@right}

\newcounter{ctr}\loop\stepcounter{ctr}\edef\X{\@Alph\c@ctr}%
	\expandafter\edef\csname s\X\endcsname{\noexpand\mathscr{\X}}
	\expandafter\edef\csname c\X\endcsname{\noexpand\mathcal{\X}}
	\expandafter\edef\csname b\X\endcsname{\noexpand\boldsymbol{\X}}
	\expandafter\edef\csname I\X\endcsname{\noexpand\mathbb{\X}}
\ifnum\thectr<26\repeat

\let\@IE\IE\let\IE\undefined
\newcommand{\IE}{\mathop{{}\@IE}\mathopen{}}
\let\@IP\IP\let\IP\undefined
\newcommand{\IP}{\mathop{{}\@IP}}
\newcommand{\Var}{\mathop{\mathrm{Var}}}
\newcommand{\Cov}{\mathop{\mathrm{Cov}}}
\newcommand{\Cor}{\mathop{\mathrm{Cor}}}
\newcommand{\Be}{\mathop{\mathrm{Be}}}
\newcommand{\bigo}{\mathop{{}\mathrm{O}}\mathopen{}}
\newcommand{\lito}{\mathop{{}\mathrm{o}}\mathopen{}}
\newcommand{\law}{\mathop{{}\sL}\mathopen{}}

\def\sump{\sideset{}{'}\sum}

\def\^#1{\relax\ifmmode {\mathaccent"705E #1} \else {\accent94 #1}\fi}
\def\~#1{\relax\ifmmode {\mathaccent"707E #1} \else {\accent"7E #1}\fi}
\def\*#1{\relax#1^\ast}
\edef\-#1{\relax\noexpand\ifmmode {\noexpand\bar{#1}} \noexpand\else \-#1\noexpand\fi}
\def\>#1{\vec{#1}}
\def\.#1{\dot{#1}}

\def\atop{\@@atop}

\newcommand{\tolaw}{\stackrel{\law}{\longto}}
\newcommand{\inj}{\mathop{\mathrm{inj}}}

\renewcommand{\leq}{\leqslant}
\renewcommand{\geq}{\geqslant}
\renewcommand{\phi}{\varphi}
\newcommand{\eps}{\varepsilon}

\newcommand{\N}{\mathop{{}\mathrm{N}}}
\newcommand{\eq}{\eqref}
\newcommand{\I}{\mathop{{}\mathrm{I}}\mathopen{}}

\newcommand\indep{\protect\mathpalette{\protect\@indep}{\perp}}
\def\@indep#1#2{\mathrel{\rlap{$#1#2$}\mkern2mu{#1#2}}}

\newcommand{\toinf}{\to\infty}

\newcommand{\longto}{\longrightarrow}

\def\sbinom#1#2{{\textstyle\binom#1#2}}
\def\parsetime#1#2#3#4#5#6{#1#2:#3#4}
\def\parsedate#1:20#2#3#4#5#6#7#8+#9\empty{20#2#3-#4#5-#6#7 \parsetime #8}
\def\moddate{\expandafter\parsedate\pdffilemoddate{\jobname.tex}\empty}

\usepackage{pgf}

\def\triangleedge{{\begin{tikzpicture}[scale=0.15]%
\node at (0,0) [draw,circle,fill=black,inner sep=0.45pt] (A) {};%
\node at (0.5,0.79) [draw,circle,fill=black,inner sep=0.45pt] (B) {};%
\node at (1,0) [draw,circle,fill=black,inner sep=0.45pt] (C) {};%
\node at (2,0)  [draw,circle,fill=black,inner sep=0.45pt] (D) {};%
\node at (2,0.79)  [draw,circle,fill=black,inner sep=0.45pt] (E) {};%
\draw (A) -- (B); %
\draw (B) -- (C); %
\draw (C) -- (A); %
\draw (D) -- (E); %
\end{tikzpicture}}}

\def\triangleappendix{{\begin{tikzpicture}[scale=0.15]%
\node at (0,0) [draw,circle,fill=black,inner sep=0.45pt] (A) {};%
\node at (0.5,0.79) [draw,circle,fill=black,inner sep=0.45pt] (B) {};%
\node at (1,0) [draw,circle,fill=black,inner sep=0.45pt] (C) {};%
\node at (2,0)  [draw,circle,fill=black,inner sep=0.45pt] (D) {};%
\draw (A) -- (B); %
\draw (B) -- (C); %
\draw (C) -- (A); %
\draw (C) -- (D); %
\end{tikzpicture}}}

\def\fourcycle{{\begin{tikzpicture}[scale=0.11]%
\node at (0,0) [draw,circle,fill=black,inner sep=0.45pt] (A) {};%
\node at (1,0) [draw,circle,fill=black,inner sep=0.45pt] (B) {};%
\node at (1,1) [draw,circle,fill=black,inner sep=0.45pt] (C) {};%
\node at (0,1) [draw,circle,fill=black,inner sep=0.45pt] (D) {};%
\draw (A) -- (B); %
\draw (B) -- (C); %
\draw (C) -- (D); %
\draw (D) -- (A); %
\end{tikzpicture}}}

\def\threepath{{\begin{tikzpicture}[scale=0.11]%
\node at (0,0) [draw,circle,fill=black,inner sep=0.45pt] (A) {};%
\node at (1,0) [draw,circle,fill=black,inner sep=0.45pt] (B) {};%
\node at (1,1) [draw,circle,fill=black,inner sep=0.45pt] (C) {};%
\node at (0,1) [draw,circle,fill=black,inner sep=0.45pt] (D) {};%
\draw (A) -- (B); %
\draw (B) -- (C); %
\draw (D) -- (A); %
\end{tikzpicture}}}

\def\triangle{{\begin{tikzpicture}[scale=0.15]%
\node at (0,0) [draw,circle,fill=black,inner sep=0.45pt] (A) {};%
\node at (0.5,0.79) [draw,circle,fill=black,inner sep=0.45pt] (B) {};%
\node at (1,0) [draw,circle,fill=black,inner sep=0.45pt] (C) {};%
\draw (A) -- (B); %
\draw (B) -- (C); %
\draw (C) -- (A); %
\end{tikzpicture}}}

\def\threestar{{\begin{tikzpicture}[scale=0.15]%
\node at (0,0) [draw,circle,fill=black,inner sep=0.45pt] (A) {};%
\node at (0.5,0.9) [draw,circle,fill=black,inner sep=0.45pt] (B) {};%
\node at (1,0) [draw,circle,fill=black,inner sep=0.45pt] (C) {};%
\node at (0.5,0.35) [draw,circle,fill=black,inner sep=0.45pt] (D) {};%
\draw (A) -- (D); %
\draw (B) -- (D); %
\draw (C) -- (D); %
\end{tikzpicture}}}

\def\twostar{{\begin{tikzpicture}[scale=0.15]%
\node at (0,0.79) [draw,circle,fill=black,inner sep=0.45pt] (A) {};%
\node at (0.5,0) [draw,circle,fill=black,inner sep=0.45pt] (B) {};%
\node at (1,0.79) [draw,circle,fill=black,inner sep=0.45pt] (C) {};%
\draw (A) -- (B); %
\draw (B) -- (C); %
\end{tikzpicture}}}

\def\twostarb{{\begin{tikzpicture}[scale=0.15]%
\node at (0,0.79) [draw,circle,fill=black,inner sep=0.45pt] (A) {};%
\node at (0.5,0) [draw,circle,fill=black,inner sep=0.45pt] (B) {};%
\node at (1,0.79) [draw,circle,fill=black,inner sep=0.45pt] (C) {};%
\draw (A) -- (B); %
\draw (C) -- (A); %
\end{tikzpicture}}}

\def\twostarc{{\begin{tikzpicture}[scale=0.15]%
\node at (0,0.79) [draw,circle,fill=black,inner sep=0.45pt] (A) {};%
\node at (0.5,0) [draw,circle,fill=black,inner sep=0.45pt] (B) {};%
\node at (1,0.79) [draw,circle,fill=black,inner sep=0.45pt] (C) {};%
\draw (B) -- (C); %
\draw (C) -- (A); %
\end{tikzpicture}}}

\def\edge{{\begin{tikzpicture}[scale=0.12]%
\node at (0,0) [draw,circle,fill=black,inner sep=0.45pt] (A) {};%
\node at (0.8,1) [draw,circle,fill=black,inner sep=0.45pt] (B) {};%
\draw (A) -- (B); %
\end{tikzpicture}}}

\makeatother

\begin{document}

\title{\sc\bf\large\MakeUppercase{Higher-order fluctuations in~dense~random~graph~models}}
\author{\sc Gursharn Kaur\footnote{Department of Statistics and Applied Probability, National University of Singapore, 6 Science Drive 2, Singapore 117546, \url{stagk@nus.edu.sg}} \and \sc Adrian R\"ollin\footnote{Department of Statistics and Applied Probability, National University of Singapore, 6 Science Drive 2, Singapore 117546, \url{adrian.roellin@nus.edu.sg}}}

\date{\itshape National University of Singapore}

\maketitle

\begin{abstract}
\noindent Our main results are quantitative bounds in the multivariate normal approximation of \emph{centred} subgraph counts in random graphs generated by a general graphon and independent vertex labels. We are interested in these statistics because they are key to understanding fluctuations of regular subgraph counts --- a cornerstone of dense graph limit theory.  We also identify the resulting limiting Gaussian stochastic measures by means of the theory of generalised $U$-statistics and Gaussian Hilbert spaces, which we think is a suitable framework to describe and understand higher-order fluctuations in dense random graph models. With this article, we believe we answer the question ``What is the central limit theorem of dense graph limit theory?''. We complement the theory with some statistical applications to illustrate the use of centred subgraph counts in network modelling.
\end{abstract}

\section{Introduction}

Since the seminal paper of \cite{Lovasz2006} on dense graph limit theory, a considerable amount of literature devoted to this topic has been published. A book-length treatment was given by \cite{Lovasz2012}, and the theory has been extended to related models, such as sparse graphs by \cite{Bollobas2009}, \cite{Borgs2014a,Borgs2014b}, \cite{Caron2017},  \cite{Borgs2017} and others, multi-graphs by \cite{Rath2012a} and \cite{Rath2012}, graphon-valued stochastic processes by \cite{Athreya2019}, and permutations by \cite{Hoppen2011} to name a few.

Much of the literature is concerned with what could generally be referred to as \emph{laws of large numbers}, where the main interest lies in describing the limiting objects upon appropriate scaling as some number~$n$ that captures the size of the model --- for example, the number of vertices of a graph --- tends to infinity. In many applications, the limiting objects are deterministic, since the randomness in the model ``averages out'', like in the case of the fraction of heads in a sequence of independent fair coin tosses. And if the limiting objects are random, then  typically because of a phenomenon related to \emph{de Finetti's Theorem} in the sense that the randomness left in the limit can be thought of as being distinct from the randomness describing the fine details of the model. 
In the case of dense graph limit theory, this phenomenon is captured by the Aldous-Hoover theory of infinite exchangeable arrays; see \cite{Diaconis2008}.

In analogy to the classical \emph{Law of Large Numbers} for sums of independent random variables, it is natural to ask about fluctuations \emph{around} the limits, which in the classical case is captured by the \emph{Central Limit Theorem}. This is of profound importance, since statistical inference is based on exactly this kind of fluctuations. But despite the large literature on dense graph limit theory, we are not aware of any attempts made to develop a higher-order fluctuation theory for random graph models, neither in the dense nor sparse regime. 

There have been some recent efforts to understand the subgraphs counts in the context of graphons and dense graph limit theory. \cite{Hladky2019} analysed the limiting distributions of $r$-clique counts of a random graph obtained through sampling from a graphon (which is our model $\IG(n,\kappa)$ below), and \cite{Chatterjee2021} generalised the results to arbitrary subgraphs. \cite{Maugis2020} analysed localised versions, where the counts are not global, but only over one specific vertex. Their results yield in essence that the scaling and limiting distribution depends on specific properties of the graphon, and this is intimately related to the work of \cite{Janson1991} on $U$-statistics. What makes subgraph counts problematic as test statistics is the fact that it is not immediately clear what is actually being tested (in other words, what aspects of the model the dominating fluctuations represent), and how different subgraph counts are related to each other, which is crucial when performing multiple test over different subgraph counts.

What we propose in this article is not to use subgraph counts as test statistics directly, but use more fundamental statistics --- centred subgraph counts --- which themselves completely determine the fluctuations of subgraph counts, which are orthogonal to each other, and which are jointly Gaussian in the limit with a straightforward covariance structure. The latter in particular allows for a proper correction when performing multiple tests. We give a rather complete description of these statistics in the dense case for models where vertices have independent labels, and conditionally on the vertex labels, edges are sampled independently of each other with probabilities given by a graphon. This model is the workhorse of dense graph limit theory, although in this article, we generalise this to sampling schemes where vertex labels need not be identically distributed. We believe the latter is an important contribution and covers the important case where vertex labels are fixed and arranged on an equally spaced lattice. 

As mentioned, the key to understanding all fluctuations is to analyse \emph{centred} subgraph counts rather than regular subgraph counts, and we are not the first to do so. Centred subgraph counts were studied in depth by \cite{Janson1994}, where the normal limits were shown using martingale methods, and by \cite{Janson1997}, who used the method of moments. \cite{Fang2012a} studied statistics similar to centred subgraph counts to construct a test whether a given graph is compatible with a constant graphon, and \cite{Bubeck2016} used centred triangle counts to construct a test for dimensionality in geometric random graphs; see also \cite{Gao2017a,Gao2017}. 

As we will argue in the next section, the mathematical framework of generalised $U$-statistics can be used to describe the fluctuations in dense graph sequences. Generalised $U$-statistics were introduced by \cite{Janson1991} to understand fluctuations of subgraph counts in the Erd\H{o}s-R\'enyi random graph and related models, and a more comprehensive treatise was given by \cite{Janson1994,Janson1997}. In particular, using the framework of \emph{Gaussian Hilbert spaces}, \cite{Janson1997} was able to describe the Gaussian limiting objects arising from generalised $U$-statistics, although his description is rather abstract and not easily interpretable in the context of dense graph limit theory. 

Our contribution is to modify the approach of \cite{Janson1997} in such a way that it becomes clearer what the limiting Gaussian Hilbert spaces are and such that it applies to non-identically distributed vertex labels, and we complement the theory with a multivariate normal approximation theorem for smooth and non-smooth test functions, which is based on Stein's method. Incidentally, none of the existing approximation theorems in the literature seem to be applicable to the present situation due to the fact that the summands in our test statistics are uncorrelated, a case that has drawn surprisingly little attention in the literature so far. Although subgraph counts can be handled using Stein's method, as was shown by \cite{Barbour1989} for smooth metrics, by \cite{Rollin2012a} for total variation and local limit metrics, by \cite{Rollin2017}, \cite{Krokowski2015} and \cite{Privault2020} for the Kolmogorov metric, centred subgraph counts, which are sums of uncorrelated but not independent random variables, cannot be handled with these approaches. The only result in this direction we are aware of is that of \cite{Fang2012a}, who considered bi-variate normal approximation for related sums of uncorrelated random variables in the case of constant graphons. 

\subsection{The basic decomposition of subgraph counts --- an example}\label{sec1-1}

\def\stk#1#2#3{\begin{tikzpicture}[scale=0.25]
	\def\s{#1}
	\def\sm{#2}
	\coordinate (e1) at (10,0);
	\coordinate (e2) at (0,10);
	\coordinate (u1) at ($(e1)+(e2)$);
	\coordinate (u2) at ($\s*(e1) + \sm*(e2)$);
	\coordinate (u3) at ($\s*(e2) + \sm*(e1)$);
	\draw (0,0) -- (e1) -- ($(e1) + (e2)$) -- (e2) -- (0,0);
	\draw ($\s*(e2)$) -- ($\s*(e2) + (e1)$);
	\draw ($\s*(e1)$) -- ($\s*(e1) + (e2)$);
	\draw ($.5*\s*(u1)$) node {$\alpha$};
	\draw ($\s*(u1) + .5*\sm*(u1)$) node {$\beta$};
	\draw ($\s*(e2) + .5*(u2)$) node {$\delta$};
	\draw ($\s*(e1) + .5*(u3)$) node {$\delta$};
	\draw ($\s*(e1)$) node[yshift=-13pt] {#3};
\end{tikzpicture}
}

\begin{figure}
\centering
\stk{.6}{.4}{$\strut\gamma$}
\caption{\label{fig1} 
The $2\times2$ graphon $\kappa$ defined in \eq{1g}.}
\end{figure}

Before elaborating on the general theory, we first illustrate what a decomposition of a subgraph into orthogonal components looks like in the simple case of a $2\times2$ block graphon, also called stochastic block model. This model is general enough to illustrate the main points, but also simple enough to work out the details, at least for simple subgraphs. 

First, fix constants $\alpha,\beta,\delta\in[0,1]$ and $\gamma\in(0,1)$. Then let $\kappa:[0,1]^2\to[0,1]$ be the graphon defined as
\ben{\label{1g}
  \kappa(x,y)
  =
  \begin{cases}
  \alpha & \text{if $x,y\leq \gamma$,} \\
  \delta & \text{if $x\leq \gamma < y$ or $y\leq \gamma < x$,} \\
  \beta & \text{if $x,y>\gamma$,} \\
  \end{cases}\qquad \qquad\text{for $x,y\in[0,1]$.}
}
This graphon is illustrated in Figure~\ref{fig1} and represents a graph with two communities with connection probability $\alpha$ and $\beta$ within the respective communities, and $\delta$ across the two communities. 

We now generate a random graph $G_n$ on $n$ vertices in the usual way. Let $U_1,\dots,U_n$ be independent random variables distributed uniformly on $[0,1]$, and conditionally on $U_i$ and $U_j$, connect vertices $i$ and $j$ with probability $\kappa(U_i,U_j)$, independently of all else. It is clear that the probability of a vertex belonging to the first community is $\gamma$ and the probability it belongs to the second community is $1-\gamma$. We denote by $Z_{i}= \I[U_i\leq \gamma]$ the indicator that vertex $i$ belongs to the first community, and by $Y_{ij}$ the indicator that $i$ and $j$ are connected. 

Now, to start with, consider the so-called \emph{edge density}
\be{
  t^{\inj}_\edge(G_n) = \frac{1}{n(n-1)}\sump_{i_1, i_2} Y_{i_1i_2},
} 
where sum ranges over all vertices and the prime in the double sum indicates exclusion of the diagonal cases $i_1=i_2$ as usual. 
With $\^Y_{ij}  = Y_{ij} - \kappa(U_i,U_j)$, it is straightforward to deduce the decomposition
\besn{\label{1ca}
  t^{\inj}_\edge(G_n) 
  & = \frac{1}{n(n-1)}\sump_{i_1, i_2} \^Y_{i_1i_2}+ \frac{1}{n(n-1)}\sump_{i_1,i_2}\bclr{\kappa(U_{i_1},U_{i_2})-\-\kappa} + \-\kappa,
}
where $\-\kappa = \IE\kappa(U_1,U_2) =  \alpha\gamma^2+2\delta\gamma(1-\gamma)+\beta(1-\gamma)^2$. Now, the second sum in~\eq{1ca} itself is a $U$-statistic, making further decomposition necessary. To this end, we write
\besn{\label{1da}
  \kappa(U_{i},U_{j})-\-\kappa 
  &= \rho_1(\^Z_{i} + \^Z_{j})+\rho_2\^Z_{i}\^Z_{j},
}
where
\ben{
  \^Z_i = Z_i -\gamma,\qquad
  \rho_1 = \alpha \gamma-\beta (1-\gamma)+(1-2 \gamma)\delta,
  \qquad
  \rho_2 =  \alpha+\beta-2\delta.
}
Using \eq{1da} on the second sum in~\eq{1ca} and a tedious exercise of adding, subtracting and rearranging terms, as well as observing that $\^Z_i^2 = (1-2\gamma)\^Z_i+\gamma(1-\gamma)$, we can write
\be{
  \sump_{i_1,i_2}\bclr{\kappa(U_{i_1},U_{i_2})-\-\kappa} 
   =  \bclr{\beta-\alpha +2n\rho_1}
      \sum_{i}\^Z_{i}+\rho_2\bbbcls{\bbbclr{\sum_{i}\^Z_{i}}^2 - n\gamma(1-\gamma)}.
}
Thus, we arrive at a complete decomposition of the form
\besn{\label{1db}
  t^{\inj}_\edge(G_n) 
  & = \-\kappa
  +\frac{2n^{1/2}\rho_1 W}{n-1}
  + \frac{\rho_2\bclr{W^2 - \gamma(1-\gamma)}}{n-1}
  + \frac{2^{1/2}V_{\edge,1}}{n^{1/2}(n-1)^{1/2}} 
  +\frac{(\beta-\alpha) W}{n^{1/2}(n-1)},
}
where
\ben{\label{1e}
W = n^{-1/2}\sum_{i} \^Z_i, 
\qquad 
V_{\edge,1} = \sbinom{n}{2}^{-1/2}\sum_{i_1 < i_2} \^Y_{i_1i_2}.
}
There are multiple reasons why this decomposition is useful. First, both $W$ and~$V_{\edge,1}$ are centred and uncorrelated random variables, and moreover, they themselves are sums of uncorrelated random variables; this is true if even the $U_i$ are not identically distributed, so long as they are independent. Hence, the variance and covariance structure is straightforward to calculate, and with the normalisations given above and assuming the $U_i$ are identically distributed, all variances are of order~$1$. Second, all quantities have Gaussian limits; for $W$, this follows easily from the classical central limit theorem, but it is also not difficult to prove for $V_{\edge,1}$ using Stein's method or the method of moments. This fact also implies the limit for $W^2$, namely a $\chi_1^2$-distribution. Third, it is now straightforward to read off the limiting behaviour of $t^{\inj}_\edge(G)$ from this decomposition (upon appropriate scaling):
\begin{enumerate}
\item $\rho_1\neq 0$:  The second term in \eq{1db} dominates and the limit is Gaussian.
\item $\rho_1 = 0$ and $ \rho_2 \neq 1$: The third and fourth terms in \eq{1db} dominate and the limit is the weighted sum of two independent random variables, one Gaussian and the other having a centred $\chi_1^2$-distribution.
\item $\rho_1 = 0$ and $ \rho_2 = 0$: The fourth term in \eq{1db} dominates and the limit is Gaussian.
\end{enumerate}
These convergence results were also obtained by \cite{Hladky2019} and \cite{Chatterjee2021}. 
Note that even in the third case, the contribution of $W$ in \eq{1db} does not vanish if $\alpha\neq\beta$, although the fluctuation only contributes to scale that is smaller than the dominating fluctuation. It is also important to recognise that this decomposition is not unique; for example, since
\be{
  \frac{1}{n-1} = \frac{1}{n}+\frac{1}{n(n-1)},
}
fluctuations at one scale can always be slightly rescaled and change the composition of fluctuations at another scale.

This simple example already reveals the subtle nature of subgraph counts under inhomogeneous sampling schemes, even for just the edge density. It also shows the main disadvantage of modelling vertex labels $U_i$ as random: In general, the subgraph counts are dominated by the group labels, rather than the randomness in the edges. This is usually not a desired property if testing a graph model against network data.

The case of triangles $t^{\inj}_\triangle(G_n) = \frac{1}{(n)_3} \sump_{i_1,i_2,i_3} Y_{i_1i_2}Y_{i_2i_3}Y_{i_1i_3}$ is much more involved and tedious to deduce, and we therefore only give the final decomposition (for multiple sums, the prime indicates exclusion of any set of indices where at least two indices coincide). We have
\be{\label{1cb}
   t^{\inj}_\triangle(G_n) = c_1 +  R_{0.5}+R_{1.0}+R_{1.5}+R_{2.0}+R_{2.5}
}
where
\ba{
 R_{0.5} &=\frac{c_2 W}{n^{1/2}}, \\
  R_{1.0}&= \frac{c_3 (W^2 -\gamma(1-\gamma))}{n-1} +\frac{c_4V_{\edge,4}+c_5(V_{\edge,2}+V_{\edge,3})+c_6 V_{\edge,1}}{n^{1/2}(n-1)^{1/2}},\\   
  R_{1.5}&=\frac{c_7W}{n^{1/2}(n-1)} + \frac{n^{1/2} c_8\bclr{W^3 - n^{-1/2} \gamma(1-\gamma)  (1-2 \gamma )}}{(n-1)(n-2)}\\
  &\quad + \frac{c_9\bclr{V_\triangle+V_{\twostar,1}+V_{\twostar,2}+V_{\twostar,3}}}{n^{1/2}(n-1)^{1/2}(n-2)^{1/2}} +\frac{c_{10}V_{\edge,1}W + c_{11}(V_{\edge,2}+V_{\edge,3})W +c_{12}V_{\edge,4}W}{(n-1)^{1/2}(n-2)},\\
  R_{2.0}&=\frac{c_{13}V_{\edge,1}+c_{14}(V_{\edge,2}+V_{\edge,3})+2 c_{15} V_{\edge,4}}{n^{1/2}(n-1)^{1/2}(n-2)}
  + \frac{c_{16}(W^2 -\gamma(1-\gamma))}{(n-1)(n-2)},\\ 
  R_{2.5}&= \frac{c_{17}W}{n^{1/2}(n-1)(n-2)},
}
with $V_{\edge,2}$ and $W$  as in \eq{1e} and with

\bgn{
  V_{\edge,2} = \sbinom{n}{2}^{-1/2}\sum_{i < j} \^Z_{i} \^Y_{ij},\quad
  V_{\edge,3} = \sbinom{n}{2}^{-1/2}\sum_{i < j} \^Z_{j} \^Y_{ij},\quad
  V_{\edge,4} = \sbinom{n}{2}^{-1/2}\sum_{i < j} \^Z_{i}\^Z_{j} \^Y_{ij},\notag\\
  V_{\twostar,1} = \sbinom{n}{3}^{-1/2}\sum_{i < j<k} \kappa(U_{i},U_{k}) \^Y_{ij}\^Y_{jk},\quad
  V_{\twostar,2} = \sbinom{n}{3}^{-1/2}\sum_{i < j<k} \kappa(U_{j},U_{k}) \^Y_{ji}\^Y_{ik},\notag\\
  V_{\twostar,3} = \sbinom{n}{3}^{-1/2}\sum_{i < j<k} \kappa(U_{i},U_{j}) \^Y_{ik}\^Y_{kj},\quad
  V_\triangle = \sbinom{n}{3}^{-1/2}\sum_{1i < j<k} \^Y_{ij}\^Y_{jk}\^Y_{ik}\notag\\[-2ex]\label{1b}
}
(the values of the constants $c_1$ to $c_{17}$ can be found in the Appendix); these results are again consistent with \cite{Hladky2019} and \cite{Chatterjee2021}. Note also that all these quantities are again uncorrelated and themselves sums of uncorrelated random variables, and they are scaled to be of order 1. We have arranged the terms so that $R_{\alpha}$ has standard deviation of order $n^{-\alpha}$. What our main result, Theorem \ref{thm1}, says is that all the quantities arising in such a decomposition are jointly close to a multivariate normal distribution that has a straightforward covariance structure, which is why we believe they are better suited for statistical applications than subgraph counts.

Note that explicit decompositions like the ones presented above are possible whenever $\kappa(U_i,U_j)$ can be written as sums and products of random variables involving the individual $U_i$ like we did in \eq{1da} for the $2\times2$ block graphon. This is possible in particular whenever $\kappa$ is piece-wise constant, that is, is of block form, and in principle one could construct an algorithm that derives such decompositions explicitly for any subgraph density and any block graphon. In general, though, \eq{1da} has to be replaced by an approximation, and that will be the content of Lemma~\ref{lem2}. 

In the case where the subgraph is not connected, one can always decompose the subgraph density into sums and products of subgraph densities of connected graphs. For example, if $\inj(F,G)$ denotes the number of injective homomorphisms from $F$ to $G$, we have
\be{
  \inj(\triangleedge,G) = \inj(\triangle,G)\inj(\edge,G)-6\inj(\triangle,G)-6\inj(\triangleappendix,G)
}
so that
\ben{\label{1a}
  t^{\inj}_{\triangleedge}(G) = \frac{n(n-1)}{(n-4)(n-5)} t^{\inj}_\triangle(G)\,t^{\inj}_\edge(G)
  - \frac{6}{(n-4)(n-5)}t^{\inj}_\triangle(G) - \frac{6}{n-5} t^{\inj}_\triangleappendix(G).
}
Hence, densities of connected subgraphs tell in essence the whole story.

\subsection{Preliminaries on dense graph limit theory}

In what follows, all graphs are assumed to be simple and finite, without loops. Consider a graph $G_n$ on the vertex set $[n]\coloneqq\{1,\dots,n\}$. For any graph $F$ on $k$ vertices, the \emph{homomorphism density of $F$ in $G_n$} is defined as 
\be{
t_F(G_n) = \frac{ \hom(F,G_n)}{n^k},
}
where $\hom(F,G_n)$ is the number of graph homomorphisms from $F$ to $G_n$. A sequence of graphs $G_1,G_2,\dots$ is called \emph{dense}, if the number of edges $e(G_n)\asymp n^2$, and it is called convergent if $\lim_{n\toinf}t_F(G_n)$ exists for all $F$. \cite{Lovasz2006} showed that if $G_1,G_2,\dots$ is a convergent dense graph sequence, then there exists a symmetric measurable function $\kappa:[0,1]^2\to[0,1]$ such that
\ben{\label{1}
  \lim_{n\toinf} t_F(G_n) = t_F(\kappa)\coloneqq \int_{[0,1]^k} \prod_{v\overset{F}{\sim} w}\kappa(x_v,x_w)dx_1\cdots dx_k,
}
where $\prod_{v\overset{F}{\sim} w}$ denotes the product over all pairs of vertices $\{v,w\}$ that are connected in $F$.
Such functions are generally referred to as \emph{graphons}, but \eqref{1} is only enough to determine $\kappa$ up to measure-preserving transformations of $[0,1]$, so the actual space of limiting objects is the equivalence class of graphons with the same values of $t_F(\kappa)$ for all $F$. What makes the representation of the limits appealing is that finite graphs can easily be embedded in the space of graphons by representing the adjacency matrix of a graph as a $0$-$1$-valued function on $[0,1]^2$ in the canonical way. If $G$ is a graph and $\kappa$ the corresponding induced graphon, it is not difficult to see that $t_F(G)=t_F(\kappa)$ for all $F$. 

In the context of graphs, the more natural objects to study are the \emph{injective} homomorphism densities, defined as 
\be{
t^{\inj}_F(G_n) = \frac{ \inj(F,G_n)}{(n)_k},
}
where $\inj(F,G_n)$ is the number of \emph{injective} homomorphisms from $F$ to $G_n$ and where $(n)_k=n(n-1)\cdots(n-k+1)$. An approximate relation between $t_F(G_n)$ and $t^{\inj}_F(G_n) $ is given by the inequality
\ben{\label{2}
\abs{t^{\inj}_F(G_n) - t_F(G_n)}\leq \frac{\binom{k}{2}}{n}.
}

Thus, $\lim t_F(G_n)=t_F(\kappa)$ if and only if $\lim t^{\inj}_F(G_n)=t_F(\kappa)$, and so from the point of view of dense graph limits, there is no difference between considering $t^{\inj}_F(G_n)$ instead of $t_F(G_n)$. However, higher-order fluctuations of these statistics are of smaller order than $n^{-1}$, and so \eqref{2} is not informative for such purposes. In this article, we will only focus on~$t^{\inj}_F(G_n)$, but results can be translated in principle via certain identities, relating the numbers of homomorphisms and injective homomorphisms although the formulas are not straightforward; see for example \cite[Section~5.2.3.]{Lovasz2012}.

\subsection{A simple (and naive) central limit theorem}

In order to motivate much of the remainder of this article, and in particular justify the expression ``higher-order'' in the title rather than just ``second-order'' as one would naturally expect from the analogy with the classical central limit theorem, we start with a heuristic analysis of the workhorse model of dense graph limit theory. Let $\kappa$ be a graphon and $U=(U_1,U_2,\dots, U_n)$ be a sequence of independent random variables, each distributed uniformly on $[0,1]$, and given $U$, let $Y_{ij}=1$ with probability $\kappa(U_i,U_j)$ and $Y_{ij}=0$ with probability $1-\kappa(U_i,U_j)$ for all $1\leq i < j\leq n$. Let $G_n$ be the graph on the vertex set $[n]$, where $i<j$ are connected if $Y_{ij}=1$ and left unconnected otherwise. We denote the resulting random graph model by $\IG(n,\kappa)$. \cite{Lovasz2006} proved the basic law of large numbers of dense graph limit theory, which states that such $G_n$ converges to $\kappa$ almost surely as $n$ tends to infinity.

The case of the Erd\H{o}s-R\'enyi random graph is the special case $\kappa\equiv p$ for some $0\leq p\leq 1$, which we assume for now. The first-order behaviour is then given by $t^{\inj}_F(G_n)\to p^{e(F)}$, where $e(F)$ is the number of edges in $F$. Moreover, it is easy to see from \eqref{1} that if~$F$ consists of connected components $F_1,\dots,F_m$, then for any graph $G$ we have 
\ben{\label{3a}
   t_F(G)= \prod_{i=1}^m t_{F_i}(G);
}
hence, it is enough to consider the fluctuations of $t_F(G)$ and $t^{\inj}_F(G)$ for \emph{connected}~$F$ (although for $t^{\inj}_F(G)$, a clean identity such as \eq{3a} that does not involve $n$ does not exist, as is apparent from \eq{1a}). Now, the second order fluctuations of $t^{\inj}_F(G)$ are not difficult to describe (see \cite{Janson1991} for the general statements and \cite{Reinert2010} for rates of convergence in some special cases). Let $K_2$ be the one-edge graph on two vertices; then, 
\ben{\label{3b}
  \Cor\bclr{t^{\inj}_{K_2}(G_n),t^{\inj}_F(G_n)}\to1, \qquad n\toinf.
}
This means that the second-order behaviour of all subgraph counts is asymptotically determined by the total number of edges. More specifically,  
\ben{\label{3}
  c_F\, n\bclr{ t^{\inj}_F(G_n)-p^{e(F)}}\approx n\bclr{t^{\inj}_{K_2}(G_n)-p}\approx \N(0,2p(1-p)),
}
where $c_F$ is some combinatorial constant depending only on $F$, and where $\N(\mu,\sigma^2)$ denotes the normal distribution with respective mean and variance. Note that \eq{3b} and \eq{3} remain true for general $F$, not just connected $F$, since even if $F$ is not connected, the quantities $t_{F_i}(G)$ in \eq{3a} are centred around positive constants, so that $t_{F}(G)$ is dominated by a linear combination of the $t_{F_i}(G)$.

Now, we can consider a more refined view of the normal distribution appearing in \eqref{3} as follows. If $\kappa_n$ denotes the $0$-$1$-graphon induced by the adjacency matrix of $G_n$, we can analyse the centred and scaled graphon measure $\^Z_n(dz)=n(\kappa_n(z)-p)dz$ for $z\in \cD_2$, and think of it as converging weakly to a white noise process  $Z_2$ living on $\cD_2 = \{(y_1,y_2)\in [0,1]^2:y_1<y_2\}$ and having infinitesimal variance $p(1-p)dy$ (see next section for exact definitions). That is, for any weight function $\phi\in L_2(\cD_2)$, 
\ben{\label{4}
  \int_{\cD_2}\phi(y)\,\hat Z_n(dy) \tolaw 
  \int_{\cD_2}\phi(y)\,Z_2(dy) \sim N(0,\norm{\phi}^2_{p,2}),
}
where $\norm{\phi}_{p,2}^2 = \int_{\cD_2} \phi(y)^2\,p(1-p)\,d y$, and this result can  easily be established for multiple $\phi$ simultaneously (see \eqref{8} for precise definition of the stochastic integral). We use the term ``white noise'' loosely here, but in the next section, we will refer to  $Z_2$ more appropriately as ``Gaussian stochastic measure'', since the term ``white noise'' has a more specific meaning in Hilda calculus; see, for example, \cite{DiNunno2009} for an excellent introduction. Further embellishments of this result could be considered, such as the convergence of the integrated process 
\be{
  \^Z_n(x,y) = \int_0^x\int_y^1 \^Z_n(du,dv), \qquad (x,y)\in \cD_2,
} 
to a corresponding Brownian sheet on $\cD_2$ (this requires a consistent ordering of the vertices, though). 

Even in this refined view, the main deficiency remains, namely that the only randomness surviving the limiting procedure is that of the number of edges, albeit now with a description at a local level. For general graphons $\kappa$, this phenomenon of loss of randomness becomes even more pronounced. As we will see, for non-constant graphons, subgraph counts are dominated by functions of the form $\sum_{i=1}^n \psi(U_i)$; that is, the randomness coming from the vertex labels dominates, and no information about the edges in the graph survives when taking limits.

While from the point of view of the classical Central Limit Theorem this could be seen as the end of the story, we have not taken into account the fact that underlying all of the graph statistics are so-called \emph{generalised $U$-statistics}. Such statistics have a much richer structure of fluctuations than sums of independent random variables. And while there is no canonical third and even higher-order fluctuation theory for sums of independent random variables, since it is not possible to make probabilistic sense out of ``subtracting the dominating effect and analyse what is left'' for sums of independent random variables without making use of signed measures, the situation for generalised (and regular) $U$-statistics is different, since fluctuations can happen simultaneously at different scales, and these fluctuations can be studied separately from one another. 

\subsection{Summary of main findings}

We now give a summary of the remainder of this article in the easier case where the vertex labels are fixed and lie on an equally spaced lattice. That is, $U_i\equiv i/n$ for $1\leq i\leq n$, and the edges $Y_{ij}$ are sampled independently with probability $\kappa(U_i,U_j)=\kappa(i/n,j/n)$, $1\leq i<j\leq n$. We will denote this random graph model by~$\IG_{\mathrm{lat}}(n,\kappa)$. The  picture that emerges from the fluctuations of subgraph counts is as follows. 

For each $k\geq 2$ and each connected graph $F$ on the vertex set $[k]$, consider the collection of \emph{centred} subgraph indicators
\ben{\label{5}
  X_{F,a}=\prod_{v\stackrel{F}{\sim}w}\bclr{Y_{a_va_w}-\kappa(a_v/n,a_w/n)},
  \qquad a \in \cI^n_k,
}
where $\cI^n_k = \{(a_1,\dots,a_k)\in \IN^k: 1\leq a_1<\cdots < a_k\leq n\}$, and the corresponding statistic
\ben{\label{6}
  W_{F} = {\textstyle\binom{n }{ k }^{-1/2}}\sum_{a\in\cI^n_k} X_{F,a}.
}
It turns out that $W_{F}$ converges to a Gaussian distribution, or more generally, the collection of random variables $X_F=(X_{F,a})_{a\in\cI^n_k}$, scaled and embedded appropriately, converges to a white noise process $Z_F$ that lives on the space
\be{
  \cD_k=\bclc{(x_1,\dots,x_k)\in[0,1]^k:x_1\leq \cdots\leq x_k}.
}
and has infinitesimal variance $\prod_{v\stackrel{F}{\sim}w}\kappa(u_v,u_w)\clr{1-\kappa(u_v,u_w)}\,du$; the case of $F=K_2$ is given in \eqref{4}. Moreover, the processes $Z_F$ turn out to be independent of each other for different $F$, and convergence holds jointly for any finite collection of $F$. 

Note that we consider ordered sums as in \eqref{6} in our main result, and the fields $Z_F$ are independent of each other even if the $F$ are isomorphic (but not identical). Hence, sums of the form 
\be{
  \sum_{a\in\cA^n_k} X_{F,a},
}
where $\cA^n_k\subset[n]^k$ is the set of $k$-tuples of pairwise different indices, can be analysed by considering ordered sums and then summing over all isomorphic copies of~$F$. 

Now, regular subgraph indicators 
\ben{\label{7}
  \prod_{v\stackrel{F}{\sim}w} Y_{vw}
}
can be approximated in $L_2$ by linear combinations of random variables of the form \eqref{5}, and in that sense, centred connected subgraph counts in \eqref{5} are really at the heart of all fluctuations of \eqref{7}. In some cases, the approximation is in fact an equality, which lead to the identify \eq{1ca} based on weighted sums of \eq{5}, leading to the quantities \eq{1e} and \eq{1b}. We will also show that the rate of convergence is $\bigo\bclr{n^{-1/2}}$ for smooth-enough test functions and of order $\bigo\bclr{n^{-1/(2(p+2))}}$ for the convex set distance, where $p$ is the maximal size of centred subgraphs considered, although the latter result is unlikely optimal.

While the collection of fields $(Z_F)_{F\in\cF}$, where $\cF$ is an enumeration of all connected finite graphs, can be thought of as the limiting object of some sort of ``centred and normalised'' graph, it is important to keep in mind that the limiting white noise fields are really just Gaussian stochastic measures, or equivalently, Gaussian Hilbert spaces, which are collections of Gaussian random variables and not objects in an actual Polish space for which we could define weak convergence. Thus, the results in this article only lay the foundations for such considerations; concretely, we establish convergence of finite dimensional distributions with rates of convergence. Further work is needed to turn this into a full-fledged notion of weak convergence. 

For the model $\IG(n,\kappa)$, where the $U_i$ are independent uniform random variables on $[0,1]$, the fields $Z_F$ need to be augmented by additional dimensions to take into account randomness of the vertex labels, as can be seen in \eq{1e} and \eq{1b}, where the $U_i$ do not just appear in the quantity $W$, but also act as weights in the remaining quantities $V_{\edge,2}$ and so forth. We will elaborate on this in more detail in Section~\ref{sec2}.

\subsection{Statistical applications}

We believe Janson and Nowicki's theory of generalised $U$-statistics along with our explicit multivariate normal approximation theorem open up new possibilities for inference in statistical network analysis. While subgraph counts have been used for inference (see, for instance, the discussion by \cite{OspinaForero2019}), we now make a few points on how centred subgraph counts could be used in statistical applications.

First, in the light of the results discussed in this article, we believe the model $\IG(n,\kappa)$ is not appropriate for statistical applications, since the randomness of the vertex labels and the randomness of the edges are conflated. It seems more natural to think of network data \emph{conditionally} on the vertex labels, which is equivalent to using the model~$\IG_{\mathrm{lat}}(n,\kappa)$.

Second, in order to calculate centred subgraph count statistics and use them for testing, the values $\kappa(U_v,U_w)$ need to be hypothesised \emph{a priori} for each pair of vertices $v$ and~$w$. As a result, a statistical procedure based on $\IG_{\mathrm{lat}}(n,\kappa)$ and centred subgraph counts to test whether the network is compatible with a specific graphon is in fact nothing but a test of whether a sequence of independent Bernoulli random variables $Y=(Y_{ij})_{1\leq i<j\leq n}$ are compatible with a specific model of their respective success probabilities $(p_{ij})_{1\leq i<j\leq n}$, and the choice of subgraphs $F\in\cF$ determines to what sort of deviations the test is sensitive to. For instance, a test based solely on the edge-count test statistic 
\be{
  T_\edge(Y)=\sigma^{-1}_{\edge}\sum_{1\leq i<j \leq n}\bclr{Y_{ij}-p_{ij}},
  \qquad \sigma^2_{\edge} = \sum_{1\leq i<j \leq n}p_{ij}\clr{1-p_{ij}},
}
is sensitive only to deviations of the overall edge density from that of the postulated model. By adding the two stars statistic
\bm{
  T_\twostar(Y)=\sigma^{-1}_{\twostar}\sum_{1\leq i<j<k\leq n}\bclr{Y_{ij}-p_{ij}}\bclr{Y_{jk}-p_{jk}},\\
   \sigma^{2}_{\twostar}= \sum_{1\leq i<j<k\leq n} p_{ij}\clr{1-p_{ij}}p_{jk}\clr{1-p_{jk}},
}
as well as the analogously defined statistics $T_\twostarb(Y)$ and $T_\twostarc(Y)$, a test will also detect deviations in the form of elevated levels of simultaneous presence or absence of edges with a common end point (leading to larger positive value of $T_\twostar(Y)$), but also the opposite, namely elevated presence of mutual suppression, where presence of one edge inhibits presence of another (leading to a larger negative value of $T_\twostar(Y)$). Correspondingly, higher order statistics yields information about presence of higher order dependencies among edges. What is noteworthy is that these statistics are very easy to calculate, and in particular the expressions for the variances are straightforward.

Third, if values for $p_{ij}$ cannot be obtained \emph{a priori}, centred subgraph counts can still serve as diagnostic tests after a model has been fitted by means of any other procedure. In such a case, these statistics can detect which aspects of the network have not been adequately captured by a model. For example, algorithms based on stochastic block models (see \cite{Funke2019} for a survey on inference methods) typically yield a community assignment for each vertex as well as connection probabilities between any two communities, and these values can serve as estimates for $p_{ij}$. One has to keep in mind, though, that in order to make a valid statistical inference such a procedure would require \emph{post-hoc} Type I error correction, since the network data has already been used to estimate the $p_{ij}$.
 
Last, an important, but difficult question is that of which centred subgraph counts should be used to determine whether a given network is compatible with a specific graphon, and this is related to the question of \emph{forcibility} of graphons; see \cite{Lovasz2011}. For example, for constant graphons, it is enough to consider edge counts and four-cycle counts; see \cite{Fang2012a} for a corresponding statistical procedure.

\begin{table}\centering
\small\setlength{\tabcolsep}{2pt}
\begin{tabular}{@{}lrrrrrrrrrr@{}}
\toprule
          & \multicolumn{10}{c}{Number of groups} \\
\cmidrule{2-11}
          & 1 & 2 & 3 & 4 & 5 & 6 & 7 & 8 & 9 & 10 \\
\cmidrule{1-11}\\[-1.5ex]
          & \multicolumn{10}{c}{Data simulated from $4\times 4$ stochastic block model}\\
\cmidrule(l){2-11}
$z_\edge$ 
  & 0.00  & 0.01 & 0.00 & \bf 0.17 & 0.08 & $-0.07$ & 0.01 & $-0.02$ & 0.10 & 0.04 \\
$z_\twostar$ 
  & 4.65 & 2.47 & 2.27 & $\boldsymbol{-0.57}$  & $-0.51$ & $-0.73$  & 0.16 & $-1.30$ & $-1.12$ & $-0.95$ \\
$z_\triangle$ 
  & $-18.57$ & $-0.38$ & 1.04 & \bf 0.03 & 0.22 & 0.15 & $-0.23$ & $-0.11$ & $-0.32$ & $-0.36$ \\
$z_\fourcycle$ & 57.36 & 2.43 & 0.77 & $\boldsymbol{-0.24}$ & $-0.07$ & $-0.04$ & $-0.33$ & $-0.33$ & $-0.69$ & $-0.60$ \\
$z_\threepath$ & $-5.39$ & 2.31 & 2.62 & $\boldsymbol{-0.93}$ & $-0.56$ & $-0.39$ & $-0.65$ & $-0.36$ & $-0.59$ & $-0.35$ \\
$z_\threestar$ & 1.90 & 2.42 & 1.55 &\bf 1.29 & 0.27 & 0.83 & 1.61 & 0.14 & 0.28 & 0.15 \\
\midrule
          & \multicolumn{10}{c}{Data set `\texttt{rfid}'}\\
\cmidrule(l){2-11}
$z_\edge$     & 0.00  & $-0.31$ &  \phantom{$-$}0.01 & $-0.33$ &  \bf \phantom{$-$}0.05 & $-0.01$ &  \phantom{$-$}0.13 & $-0.07$ & $-0.17$ &  \phantom{$-$}0.10 \\
$z_\twostar$  & 71.48 & 19.49 & 8.72 &  9.32 & \bf 9.87 & 6.25   & 1.02 & $-0.15$  & 2.31 & 0.68 \\
$z_\triangle$ & 11.32 & 1.49  & 1.53 &  4.40 & \bf 7.96 & 8.24   & 8.20 & 8.25 & 6.73 & 6.33 \\
$z_\fourcycle$ & 136.27 & 30.38 & 22.25 & 17.06 & \bf 13.43 & 13.15 & 12.31 & 12.44 & 10.94 & 10.86 \\
$z_\threepath$ & 44.32 & 1.11 & $-1.74$ & $-1.85$ & $\boldsymbol{-0.41}$ & 0.50 & $-1.24$ & $-1.32$ & $-0.67$ & $-1.18$ \\
$z_\threestar$ & 44.23 & $-3.36$ & 4.55 & 5.87 & \bf 7.85 & 2.97 & $-0.52$ & 0.28 &  $-0.67$ & $-0.57$ \\
\bottomrule
\end{tabular}
\caption{\label{tab1} Results of centred and standardized subgraph counts for fitted models using the function \texttt{BM\_Bernoulli} from the \texttt{R}-package `\texttt{blockmodels}'. The numbers reported are defined in \eq{100b}. The first data set is simulated from a $4\times 4$ stochastic block model with 200 vertices, where each group has 50 vertices. The second data set is the hospital encounter network `\texttt{rfid}' from the \texttt{R}-package \texttt{igraphdata}, and consists of 75 vertices representing hospital staff along with encounter counts for each pair of staff. The bold columns represent the estimated number of groups as recommended by the function \texttt{BM\_Bernoulli} using the \emph{Integrated Classification Likelihood (ICL)} criterion proposed by \cite{Biernacki2000}.}
\end{table}
To illustrate how centred subgraph counts can be used in actual applications, we have analysed two data sets of small networks. The first is a simulated network with 200 vertices, drawn from a $4\times4$ stochastic block model with connection probabilities given by the matrix
\be{
K=\begin{pmatrix}
0.45 & 0.34 & 0.82 & 0.60 \\
0.34 & 0.70 & 0.98 & 0.57 \\
0.82 & 0.98 & 0.03 & 0.82 \\
0.60 & 0.57 & 0.82 & 0.25
\end{pmatrix}.
}
Each group had 50 vertices assigned to it, after which the connections were sampled based on the probabilities $K$ and the respective groups two vertices belonged to. We then used the function `\texttt{BM\_bernoulli}' from the \texttt{R}-package `\texttt{blockmodels}' which, for each given number of groups, does both assignment of group labels to vertices  (also called `clustering', `community detection' or `community recovery') and estimation of connection probabilities $\^K$. From this, the individual connection probabilities $\^p_{ij}$ were be derived. The package \texttt{blockmodels} uses the \emph{Integrated Classification Likelihood (ICL)} criterion proposed by \cite{Biernacki2000} to choose the optimal number of groups, but we have done the centred subgraph count analysis for all group sizes from 1 to 10; the results are given in Table~\ref{tab1}. With $y=(y_{ij})_{1\leq i<j\leq n}$ denoting the edge indicators and $(\^p_{ij})_{1\leq i<j\leq n}$ denoting the estimated connection probabilities (which are a function of $y$), let $\^y_{ij} = y_{ij}-\^p_{ij}$, the numbers reported in Table~\ref{tab1} are defined as
\besn{\label{100b}
  z_\edge(y) & =\^\sigma^{-1}_{\edge}(y)\sum_{1\leq i<j \leq n}\^y_{ij},
  \qquad
  z_\triangle(y)  = \^\sigma^{-1}_{\triangle}(y) \sum_{1\leq i<j<k \leq n} \^y_{ij}\^y_{jk}\^y_{ik},\\
  z_\twostar  & = \^\sigma^{-1}_{\twostar}(y) \sum_{1\leq i<j<k \leq n}\bclr{\^y_{ij}\^y_{jk}+\^y_{ji}\^y_{ik}+\^y_{ik}\^y_{kj}},
}
where
\ba{
   \^\sigma^2_{\edge}(y) & = \sum_{1\leq i<j \leq n}\^p_{ij}\clr{1-\^p_{ij}},\quad
   \^\sigma^2_{\triangle}(y)  = \sum_{1\leq i<j<k \leq n} \^p_{ij}\clr{1-\^p_{ij}}\^p_{jk}\clr{1-\^p_{jk}}\^p_{ik}\clr{1-\^p_{ik}}\\
   \^\sigma^2_{\twostar}(y) & = \sum_{1\leq i<j<k \leq n}\bclr{
   \^p_{ij}\clr{1-\^p_{ij}}\^p_{jk}\clr{1-\^p_{jk}} + \^p_{ji}\clr{1-\^p_{ji}}\^p_{ik}\clr{1-\^p_{ik}} \\[-2ex]
   &\kern20em+\^p_{ik}\clr{1-\^p_{ik}}\^p_{kj}\clr{1-\^p_{kj}}};
}
the quantities $z_\fourcycle$, $z_\threepath$ and $z_\threestar$ are defined analogously.
For this simulated network data, it is evident that the models fitted by \texttt{BM\_bernoulli} for four or more groups are consistent with the three centred subgraph count statistics reported, and that they are not consistent when only one, two or three groups are hypothesised. 

The second data set is from a hospital, where 75 individuals of the staff were equipped with devices to record encounters between these individuals whenever the devices were in close proximity to each other. Since the network is edge-weighted, where each weight represents the number of encounters recorded over the time period of the experiment, we have converted this to a simple unweighted network. An edge is present between two individuals if they had at least one encounter. It is clear that for this real data, the fitted stochastic block models do not capture higher order dependence well. Even when allowing the clustering algorithm dividing the vertices into ten groups, the dependence captured by centred triangles is not what one would expect from a stochastic block model.

\section{Subgraph counts and generalised \texorpdfstring{$U$}{U}-statistics}\label{sec2}

In this section we review and discuss material from \cite{Janson1991} and \cite{Janson1994,Janson1997} in order to show that the existing literature on generalised $U$-statistics does provide a suitable framework to describe the fluctuations arising in standard dense graph models. The material in this section is not strictly necessary to state and prove the main results in Section~\ref{sec3}, but it gives the motivating context and associated general limit theory.

\subsection{Gaussian Hilbert Spaces}

We follow \cite{Janson1997} in essence. While Gaussian Hilbert spaces will serve as a form of limiting objects, it is important to keep in mind that at this level of abstraction, Gaussian Hilbert spaces are just collections of Gaussian random variables, and there is not really a single object taking values in a single space. Although, for instance, Brownian motion indexed by time can be seen as a Gaussian Hilbert space, it comes with additional properties such as almost-sure path-wise continuity, which is a statement about the joint distribution of uncountably many of the variables and goes beyond the general theory discussed here.

Now, let $H$ be a Hilbert space, where we denote the inner product by $\ip{\cdot,\cdot}_H$ and the resulting norm by $\norm{h}_H:=\sqrt{\ip{h,h}_H}$, although we will drop the dependence on $H$ for norms and inner products if there is no ambiguity. A \emph{Gaussian Hilbert space} indexed by $H$ is a collection of centred Gaussian variables $(Z_h)_{h\in H}$ defined on a common probability space $(\Omega, \cF, \IP)$ such that 
\be{
  \Cov(Z_h,Z_{h'})=\IE\clc{Z_hZ_{h'}} = \ip{h,h'}_H, \quad h,h'\in H.
}
Clearly, $\IE Z_h^2 = \Var Z_h = \norm{h}^2_H$. It is known that such a family can be constructed for every Hilbert space in such a way that, if $h_n \to h$ in $H$ as $n\toinf$, then $Z_{h_n}\to Z_h$ in $L_2(\Omega,\cF,\IP)$. Moreover, any countable collection of the $Z_h$ of a Gaussian Hilbert space is jointly Gaussian.

\paragraph{Gaussian stochastic measures and stochastic integrals.} Let $(M,\cM,\mu)$ be a measure space, and consider the Hilbert space $L_2(M)$ (we drop the $\sigma$-algebra and measure from the notation if it does not cause ambiguity). A Gaussian Hilbert space $(Z_\phi)_{\phi\in L_2(M)}$ can be interpreted as a \emph{Gaussian stochastic integral} on $M$ by setting
\ben{\label{8}
  \int_M \phi(x)\, Z(dx) := Z_\phi,\qquad \phi\in L_2(M).
}
Indeed, the family of random variables defined by $Z(A):= Z_{I_A}$, where $A\in\cM$ with $\mu(A)<\infty$ so that the indicator function $I_A$ is in $L_2(M)$, defines a \emph{Gaussian stochastic measure} $Z$ on $M$, which has the  following properties:
\begin{compactenum}[$(i)$] 
\item if $A\in \cM$ with $\mu(A)<\infty$, then
\be{
  Z(A)\sim N(0,\mu(A));
} 
\item if $A_1,A_2,\ldots\in M$ are disjoint sets with $\mu(A_i)<\infty$ for all $i\geq 1$, then the random variables $Z(A_1),Z(A_2),\ldots$ are mutually independent and 
\be{
  Z\bbclr{\,\bigcup_{i\geq 1}A_i} = \sum_{i\geq 1} Z(A_i)
}
(note that convergence on the right hand side is in $L_2$, but since the summands are independent, it is also almost sure by Kolmogorov's three-series theorem). 
\end{compactenum}
This justifies the notation $\int \phi\, dZ$ in \eqref{8}, since any Gaussian stochastic measure on $M$ in turn uniquely defines a Gaussian stochastic integral via the standard procedure of approximating functions in $L_2(M)$ via simple functions and taking closure. Note that a Gaussian stochastic measure can be loosely interpreted as white noise, but we will avoid this terminology for the remainder of this article for the reasons given in the previous section.

We can extend the single stochastic integral to a multiple stochastic integral
\ben{\label{10}
  \int_{M^k}\phi(x)\, Z^k(dx),  \qquad \phi\in L_2(M^k,\mu^k),
}
where $\mu^k$ is the usual product measure on the product sigma-algebra $\cM^{\bigotimes k}$. To this end, let $A_1,\dots,A_n\subset M$ be measurable and pairwise disjoint, and consider simple functions of the form
\ben{\label{11}
  \phi(x) = \sum_{i_1,\dots,i_k=1}^n \phi_{i_1,\dots,i_k} \I[x_1\in A_{i_1},\dots,x_k\in A_{i_k}],
} 
where $\phi_{i_1,\dots,i_k}$ vanishes whenever any two of the indices coincide. For such functions, the multiple integral can be defined as
\be{
  \int_{M^k}\phi(x)\, Z^k(dx) = \sum_{i_1,\dots,i_k=1}^n \phi_{i_1,\dots,i_k} Z( A_{i_1}) \cdots Z(A_{i_k}),
}
and the general case  $\phi\in L_2(M^k)$ can be obtained by approximating such functions by functions of the form \eqref{11}; we refer to \cite{Nualart2006} for details. The integral \eqref{10} turns out to be an element of the $k$th \emph{Wiener Chaos} $\cH_k$, which is the $L_2$-closure of the space generated by the random variables $\{H_k(Z_h); h\in H,\norm{h}_H=1\}$, where $H_k$ is the $k$th Hermite polynomial. 

\subsection{Gaussian limits related to sums of independent random variables}

Before detailing on the results known for generalised $U$-statistics, it is illuminating to briefly review the different types of results known for independent random variables, and how these results can be formulated in the framework of Gaussian Hilbert spaces. 

In what follows, let $X_1,X_2,\dots,$ be independent and identically distributed random variables with $\IE X_1 = 0$ and $\Var X_1 = 1$. 

\paragraph{Central Limit Theorem and Donsker's theorem.} Let $H=\IR$; the corresponding Gaussian Hilbert space can be simply constructed by taking a standard Gaussian variable $Z_1$ and letting $Z_c = cZ_1$ for $c\in \IR$. The standard CLT then yields
\ben{\label{12}
  \frac{1}{n^{1/2}}\sum_{i=1}^n cX_i \tolaw Z_c,\qquad c\in \IR.
}
We can generalise \eqref{12} and replace the constant $c$ on the left hand side by a general weight function. To this end, consider the Hilbert space $H=L_2([0,1])$ with the usual inner product $\ip{\phi_1,\phi_2}=\int_0^1 \phi_1(x)\phi_2(x) dx$, and let $(Z_\phi)_{\phi\in L_2([0,1])}$ be a corresponding Gaussian Hilbert space. For given $\phi\in L_2([0,1])$, which we assume to be continuous almost everywhere to avoid certain technical difficulties which are irrelevant for this discussion, Donsker's theorem yields
\ben{\label{13}
   \frac{1}{n^{1/2}}\sum_{i=1}^n \phi(i/n)X_i \tolaw \int_{[0,1]} \phi(x)\,Z(dx),\qquad \phi\in  L_2([0,1]),
}
and this holds jointly for any finite collection of such $\phi$. Moreover, \eqref{12} follows from \eqref{13} if we choose $\phi\equiv c$, where $c\in\IR$. It is important to stress that Donsker's theorem gives a stronger result than that. In fact, if we take $\phi_t(x)=\I[x\leq t]$, where $0\leq t\leq 1$, we can construct the Gaussian Hilbert space in such a way that the process $(Z_{\phi_t})_{0\leq t\leq 1}$ is almost surely \emph{continuous in $t$}, so that this process can be identified with standard Brownian motion $B_t = Z_{\phi_t}$ for $0\leq t\leq 1$. And so, what Donsker's theorem actually yields is that 
\ben{\label{14}
  \bbbclr{\frac{1}{n^{1/2}}\sum_{i=1}^{\floor{nt}} X_i}_{0\leq t\leq 1} \tolaw (B_t)_{0\leq t\leq 1},
}
where weak convergence is with respect to the Skorohod topology (or uniform topology if the process on the left hand side of \eqref{14} is interpolated between jumps).

\paragraph{$\boldsymbol{U}$-statistics.} Before turning to $U$-statistics, we first consider real-valued functions of the $X_i$. To this end, we may assume that the $X_i$ take values in a general measure space $\cS$, and we denote the distribution of $X_i$ by $\mu$. Consider the Hilbert space $H = L_2( [0,1]\times\cS, dt\times \mu)$  with the canonical inner product that satisfies $\ip{\phi_1\psi_1,\phi_2\psi_2}=\ip{\phi_1,\phi_2}\ip{\psi_1,\psi_2}_\mu$, where $(\phi\psi)(x,y):=\phi(x)\psi(y)$. Denoting by $L_2^\circ(\cS)$ the set of functions $\psi\in L_2(\cS)$ with $\IE \psi(X_1)=0$, and assuming again that $\phi$ is continuous almost everywhere, it can be shown that
\bmn{\label{15}
   \frac{1}{n^{1/2}}\sum_{i=1}^n \phi(i/n)\psi(X_i) \tolaw 
   \int_{[0,1]\times \cS} \phi(t)\psi(x)\,Z(dt, dx),\\
    \phi\in L_2([0,1]),\,\psi\in L_2^\circ(\cS).
}
Again, this statement is also true jointly for any finite collection of $(\phi_i,\psi_i)$. Note that $Z$ has infinitesimal variance $\mu(dx)$, since the space $L_2^\circ(\cS)$ comes with inner product $\angle{\psi_1,\psi_1}=\int_{\cS}\psi_1(x)\psi_2(x)\,\mu(dx)$, and the quantity $Z(dt\times dx)$ can be loosely interpreted as ``the normalised number of times the value $dx$ has been observed among the indices $dt$''.  Note also that restricting $\psi$ to be in $L_2^\circ(\cS)$ and not in $L_2(\cS)$ is necessary, since the measure $Z$ has more degrees of freedom than the finite-$n$ system, and thus cannot be fully observed. For example, if $X_i\sim\Be(p)$, then $\cS=\{0,1\}$, and all functions $\psi\in L_2^\circ(\cS)$ are multiples of the function $\psi(0)=-p$ and $\psi(1)=1-p$. However, the variables $Z_0=Z([0,1]\times\{0\})$ and $Z_1=Z([0,1]\times\{1\})$, while constructed to be independent, cannot be observed individually --- only their weighted sum $-pZ_0+(1-p)Z_1$ can be observed. This stems from the fact that the number of $X_i$ with value $1$ must equal $n$ minus the number of $X_i$ with value $0$, and in this sense, the Gaussian Hilbert space $L_2([0,1]\times \cS)$ is slightly too big.
 
The result for $U$-statistics can be stated without introducing a new Gaussian Hilbert space --- we only need multiple integrals over the same space. To this end, let 
\ben{\label{16}
  L_2^{\circ}(\cS^k) = \bbbclc{\psi\in L_2(\cS^k)\,:\,  \int_{\cS} \psi(x_1,\dots,x_k)\,\mu(dx_i)=0,\,1\leq i\leq k,\,(x_j)_{j\neq i}\in \cS^{k-1}}.
}
For $a=(a_1,\dots,a_k)\in\cI^n_k$, write $a/n \coloneqq(a_1/n,\dots,a_k/n)$ and $X_a=(X_{a_1},\dots,X_{a_k})$. Then, for almost everywhere continuous $\phi\in L_2(\cD_k)$,
\bmn{\label{17}
   \frac{1}{n^{k/2}}\sum_{a\in\cI^n_{k}} \phi(a/n)\psi(X_a) \tolaw \int_{\cD_k\times \cS^k} \phi(t)\psi(x)\,Z^k(dt, dx),\\[-1ex]
   \phi\in L_2(\cD_k),\,\psi\in L_2^{\circ}(\cS^k).
}
Again, this statement is true jointly for any finite collection of $(\phi_i,\psi_i)$, even with different~$k$; see \cite[Theorem~11.16]{Janson1997}. Of course \eqref{15} is just a special case of \eqref{17}, but \eqref{17} \emph{follows} in essence from \eqref{15} and the continuous mapping theorem. General functions $\psi\in L_2(\cS^k)$ can be decomposed into orthogonal elements $\psi_i\in L_2^{\circ}(\cS^i)$, and a corresponding limit result then follows from \eqref{17}, depending on the lowest-order, non-vanishing element $\phi_i$ (which in turn also determines the correct scaling to obtain a non-trivial limit); see \cite[Theorem~11.19]{Janson1997}. 

\paragraph{Generalised $\boldsymbol{U}$-statistics.}

The final extension we consider are generalised $U$-statistics, which were introduced by \cite{Janson1991}. To this end, assume the $X_i$ now take values in a space $\cS_1$ with distribution $\mu_1$, and let $(Y_{ij})_{1\leq i<j}$ be independent and identically distributed random elements taking values in a space $\cS_2$ with distribution $\mu_2$. For $a\in\cI^n_k$, we define $X_a$ as before and we let $Y_a = (Y_{a_ia_j})_{1\leq i<j\leq k}$. We now consider functions defined on the space 
\be{
  \cT_k = \cS_1^k\times \cS^{\binom{k}{ 2}}_2,\qquad \text{with measure $\mu_1^k\times \mu_2^{\binom{k}{ 2}}$,}
}
where it is understood that $\cT_1 = \cS_1$. For any function $\psi\colon \cT_k\to\IR$, we will write 
\be{
  \psi(X_a,Y_a) = \psi(X_{a_1},\dots,X_{a_k},Y_{a_1a_2},\dots,Y_{a_{k-1}a_k}).
}
Let now
\be{
  L_2^\circ(\cT_1) = \bclc{\psi\in L_2(\cT_1): \IE \psi(X_1)=0}.
}
For $k> 1$, let
\be{
  \cF^k_{-l} = \sigma\clr{X_1,\dots,X_k}\vee\sigma\bclr{Y_{ij}:\text{for all $1\leq i<j\leq k$ with $l\notin\{i,j\}$}}
}
Then, define
\ben{\label{18}
  L_2^\circ(\cT_k) = \bclc{\psi\in L_2(\cT_k): \text{$\IE\bclc{
  \psi\clr{X,Y}
  \given \cF^k_{-l}}=0$ for all $1\leq l\leq k$}}.
}
In words, $L_2^\circ(\cT_k)$ consists of all those functions that, for every $1\leq l\leq k$, vanish when being simultaneously integrated over all $Y_{ij}$ with $i=l$ or $j=l$. Then, with $Z_k$ a Gaussian stochastic measure on $L_2(\cD_k\times \cT_k)$,
\besn{\label{19}
   &\frac{1}{n^{k/2}}\sum_{a\in\cI^n_k} \phi(a/n)\psi(X_a,Y_a) \tolaw \int_{\cD_k\times \cT_k} \phi(t)\psi(x,y)\, Z_k(dt, dx , dy),\\
   &\kern20em \phi\in  L_2(\cD_k),\,\psi\in L_2^\circ(\cT_k)
}
(here, $t$ and $x$ are $k$-dimensional vectors, while $y$ is a  $\binom{k}{2}$-dimensional vector); see \cite[Theorem~11.28]{Janson1997} for general $\psi\in L_2(\cT_k)$ via orthogonal decomposition. 

\subsection{Application to centred subgraph counts}

We first apply \eqref{19} to centred subgraph counts of $\IG(n,\kappa)$. However, since the $Y_{ij}$ in $\IG(n,\kappa)$ are not independent of each other, we need to resort to an auxiliary representation. Let $(U_i)_{i\geq 1}$ and $(V_{ij})_{i,j\geq 1}$ be independent random variables uniformly distributed on $[0,1]$, hence $\cS_1= \cS_2 = [0,1]$, endowed with the Lebesgue measure. We construct a graph $G_n$ on the vertex set $[n]$ by connecting vertices $i$ and $j$ if $V_{ij}\leq \kappa(U_i,U_j)$. For a given connected graph $F$ on $k$ vertices, we consider the function
\ben{\label{20a}
  \psi_F\bclr{u,v}= \psi(u)\prod_{i\stackrel{F}{\sim} j}\bclr{ \I[v_{ij}\leq \kappa(u_i,u_j)]-\kappa(u_i,u_j)}, 
  \qquad u\in[0,1]^k,\,v\in[0,1]^{\binom{k}{ 2}}.
}
It is easy to verify that $\psi_F\in L_2^\circ(\cT_k)$, and hence
\besn{\label{20}
  &\frac{1}{n^{k/2}}\sum_{a\in\cI^n_k}\phi(a/n)\psi(U_a)\prod_{i\stackrel{F}{\sim} j}\bclr{ \I[V_{a_ia_j}\leq \kappa(U_{a_i},U_{a_j})]-\kappa(U_{a_i},U_{a_j})} \\
  &\quad\tolaw 
  \int_{\cD_k\times[0,1]^k\times[0,1]^{\binom{k}{ 2}}}\phi(t)\psi(u) \prod_{i\stackrel{F}{\sim} j}\bclr{ \I[v_{ij}\leq \kappa(u_i,u_j)]-\kappa(u_i,u_j)} \, Z_k(dt, du, dv).
}
Two comments are in place. First, the quantity $Z_k(dt, du, dv)$, in particular the $dv$-part, does not admit an intuitive interpretation, since the uniform random variables $V_{ij}$ are only used as an auxiliary tool to represent subgraph counts as generalised $U$-statistics. Evaluating the stochastic integral in \eq{20} with respect to $dv$, it is not difficult to show that integral can be written as
\be{
\int_{\cD_k\times[0,1]^k}\phi(t)\psi(u)  \, Z_k(dt, du),
}
where $Z_k$ now is a Gaussian stochastic measure on the space
\be{
  L_2\bbbclr{\cD_k\times[0,1]^k,dt\times \prod_{i\overset{F}{\sim}j} \kappa(u_i,u_j) \clr{1-\kappa(u_i,u_j)} du}.
}
Then, instead of \eq{20a}, we will consider functions of the form
\be{
  \psi_F(u,y) = \psi(u)\prod_{i\stackrel{F}{\sim} j}\bclr{ y_{ij}-\kappa(u_i,u_j)},
  \qquad u\in [0,1]^k,\,y\in [0,1]^{\binom{k }{ 2}},\,\psi\in L_2([0,1]^k),
} 
which allows to avoid the auxiliary representation via the $V_{ij}$ and use the $Y_{ij}$ directly, despite their dependence, and also allows to consider weighted edges. Hence, what we will show in the main section is that for $U$ as before, and random variables $Y_{ij}$, which are conditionally independent given $U$ and which satisfy~$\IE\clc{Y_{ij}\given U}=\kappa(U_{i},U_j)$,
\bmn{\label{21}
  \frac{1}{n^{k/2}}\sum_{a\in\cI^n_k}\phi(a/n)\psi(U_a)\prod_{i\stackrel{F}{\sim} j}\bclr{Y_{a_ia_j}-\kappa(U_{a_i},U_{a_j}) } \\
  \qquad\tolaw 
  \int_{\cD_k\times[0,1]^k}\phi(t)\psi(u)  \, Z_k(dt, du).
}
Note that the measure $Z_k$ is homogeneous over $\cD_k$; this is because the $U_i$ `average out' the differences in the variances of the $Y_{ij}$, so that the points in $\cD_k$ only see the combined variance effect across all $Y_{ij}$. This is different in $\IG_{\mathrm{lat}}(n,\kappa)$, which is not vertex-exchangeable; see Remarks \ref{rem1} and \ref{rem2} for further discussion.

Second, \eqref{20} does not cover the case where the $U_i$ are not identically distributed. This is important in particular for the model $\IG_{\mathrm{lat}}(n,\kappa)$, but our results hold in greater generality.

\subsection{An orthogonal decomposition of subgraph counts}

We now discuss how \eqref{21} can be used to understand fluctuations of subgraph counts of $\IG(n,\kappa)$, and so fix a graphon $\kappa$, let $U=(U_1, \dots, U_n) $ be a sequence of independent random variables uniformly distributed on $[0,1]$, let $Y=(Y_{ij})_{1\leq i<j\leq n}$ be conditionally independent given $U$ and distributed as before, and construct $G_n$ on $n$ vertices as before. Let $F$ be a graph on the vertex set $[k]$ and write  
\ben{\label{22}
  t^{\inj}_F(G_n) = \frac{1}{(n)_k}\sum_{a\in\cA^n_k}\prod_{i\stackrel{F}{\sim} j}Y_{ij};
}
without loss of generality, we may assume that $F$ has no isolated vertices. 
Now, the following lemma states that $t^{\inj}_F(G_n)$ can be decomposed into a sum of mostly uncorrelated centred subgraph counts, weighted by functions of the vertex labels. Some notation is needed. For a graph $H$ we denote by $\abs{H}$ the number of vertices in $H$, and for two (vertex-labelled) graphs $H$ and $H'$, we denote by $H\cup H'$ the graph where two vertices are connected if they are connected in at least one of $H$ and $H'$. For a subgraph $H\subseteq F$ and a subset $A\subseteq[k]$, we denote by $H\cup A$ the graph obtained by interpreting $A$ as the empty graph on the vertex set $A$. Moreover, we denote by $H_P$ the unique graph on the vertex set $\{1,\dots,\abs{H}\}$ that is isomorphic to $H$ and preserves the ordering of the vertex labels, and we denote by $H\subseteq' F$ that $H$ is a subgraph of $F$ and that it has no isolated vertices. Let
\ben{\label{23}
  \vartheta_H(u,y) = \prod_{i\stackrel{H}\sim j}\bclr{y_{ij}-\kappa(u_i,u_j)}.
}

\begin{lemma}\label{lem1} Let $F$ be a graph on the vertex set $[k]$ without isolated vertices. Then there are functions $\psi_{H,A}\in L_2^\circ\bclr{[0,1]^{\abs{A}}}$ for $H\subseteq'F$ and $A\subset[k]$, such that
\ben{\label{24}
  t^{\inj}_F(G_n) = \sum_{H\subseteq'F}\sum_{A\subseteq[k]}r_{H,A}(U,Y),
}
where, with $l=\abs{H\cup A}$,
\be{
  r_{H,A}(u,y) = \frac{1}{(n)_l}\sum_{a\in\cA^n_{l}} \psi_{H,A}(u_{a_{l-\abs{A}+1}},\dots,u_{a_l})\,\vartheta_{H_P}(u_{a_1},\dots,u_{a_{\abs{H}}}),
}
and such that the following holds: If $H$ and $H'$ are not isomorphic or if $\abs{A}\neq\abs{A'}$, then 
\be{
  \Cov\bclr{r_{H,A}(U,Y),r_{H',A'}(U,Y)} = 0,
} 
and if $\psi_{H,A}\not\equiv 0$, then, again with $l=\abs{H\cup A}$,
\be{
  \Var r_{H,A}(U,Y) \asymp n^{-l}.
}
\end{lemma}

The key of this decomposition is that terms with different scalings are uncorrelated, so that we can separate the different orders of fluctuations of $t^{\inj}_F(G_n)$. However, whether $r_{H,A}(U,Y)$ has a normal limit or not, depends on $H$ and $A$. Simply put, if $H\cup A$ is connected (for which it is necessary that $A\subset H$), the limit is normal, otherwise the limit is an element from a higher-order Wiener chaos. However, since each Wiener chaos itself is obtained by taking products of the underlying Gaussian Hilbert space and taking limits, we can decompose $r_{H,A}$ further, but only in an approximate sense.  

The following lemma makes this precise and states that the statistics on the right hand side of \eqref{24} can be approximated in $L_2$ by products and sums of simpler statistics to any prescribed level of accuracy. 

\begin{lemma}\label{lem2} Let $H\subseteq' F$, and let $A\subset[k]$. Let $\psi_{H,A}$ and $r_{H,A}$ be as in Lemma~\ref{lem1}. Denote by $C_1,\dots,C_r$ the connected components of $H\cup A$ and $k_1,\dots,k_r$ their respective sizes, and assume $r\geq 2$ (note that $k_1+\cdots+k_r=l$). For each $j$, let $C_j'$ be an isomorphic copy of $C_j$ on $[k_j]$. Then, for each $\eps>0$, there exists $N$ and there exist functions $\psi_{i,j}\in L_2^\circ([0,1]^{k_j})$ for $1\leq i\leq N$ and $1\leq j\leq r$, such that
\ben{\label{25}
  \IE\bbbclr{r_{H,A}(U,Y) - \sum_{i=1}^{N} \prod_{j=1}^r \frac{1}{(n)_{k_j}}\sum_{a\in\cA^n_{k_j}}\psi_{i,j}(U_a)
  \vartheta_{C_j'}(U_a,Y_a)}^2\leq \frac{\eps}{n^l},
}
for all $n$.
\end{lemma}

In other words, the standardised statistics $n^{l/2}r_{H,A}$ can be approximated in $L_2$ by products and sums of centred \emph{connected} subgraphs counts uniformly in $n$, and we will show that these statistics themselves all have Gaussian limits. While the overall quality of approximation of $t^{\inj}_F$ is only of order $n^{-2}$ in general, the important point here is of course that the fluctuations at different scalings are uncorrelated and become independent in the limit.

\section{Main result}\label{sec3}

We are now ready to formulate our main result, which provides bounds on the multivariate normal approximation of sums as they appear in \eqref{25}. In order to have a cleaner framework, our result will be formulated for sums over the index set $\cI^n_k$, which makes all summands  uncorrelated, but sums over $\cA^n_k$ as they appear in \eqref{25} can of course be easily computed from sums over $\cI^n_k$.

\subsection{Gaussian approximation of centred connected subgraph counts}

Let $n\geq 1$, let $\kappa$ be a graphon, and assume $\kappa\not\equiv 0$ and $\kappa\not\equiv 1$. Let $U=(U_v)_{v\in[n]}$ be independent (but not necessarily identically distributed) random variables, and given $U$, let $(Y_{vw})_{1\leq v<w\leq n}$ be random variables that are conditionally independent given~$U$ and that satisfy $\IE\clc{Y_{ij}\given U}=\kappa(U_i,U_j)$. Recall that we set $U_a=(U_{a_1},\dots,U_{a_k})$ for $a\in \cI^n_{k}$. Let $d\geq1$, and for each $1\leq i\leq d$, let $F_i$ be a connected graph on the vertex set $[k_i]$, where $k_i\geq 1$, let $\phi_i\in L_2(\cD_{k_i})$ and $\psi_i\in L_2([0,1]^{k_i})$  (note that here we do not require that $\psi_i\in L_2^\circ([0,1]^{k_i})$, since centring is either done explicitly below if $k_i=1$, or else is not necessary since the centred subgraphs provide the centring). For $1\leq i\leq d$, define 
\ben{\label{26}
W_{i} = 
  \begin{cases}
    {n}^{-1/2}  \displaystyle   \sum_{a=1}^n \phi_i(a/n)\bclr{\psi_i( U_{a})-\IE\psi_i( U_{a})} 
    &\text{if $k_i=1$,}\\[4ex]
    \displaystyle{\textstyle\binom{n}{ k_i}}^{-1/2} \displaystyle  \sum_{a\in \cI^n_{k_i}} \phi_i(a/n)\,\psi_i(U_a) \prod_{v\overset{F_i}{\sim}w}\clr{Y_{a_va_w}-\kappa(U_{a_v},U_{a_w})}
    &\text{if $k_i\geq 2$,}
  \end{cases}
}
and let $W=(W_{1},\dots,W_{d})$. Then, for $\Sigma = (\sigma_{ij})_{1\leq i,j\leq d}=\Var W$, we have   
\besn{\label{27}
\sigma_{ij} =
  \begin{cases} 
    \displaystyle n^{-1}\sum_{a=1}^n \phi_i(a/n)\phi_j(a/n) \Cov\bclr{\psi_i( U_{a}),\psi_j( U_{a})}
      &\text{if $k_i=k_j=1$,}\\[5ex]
     {\textstyle\binom{n}{ k_i}}^{-1}\displaystyle \sum_{a\in \cI^n_{k_i}}\phi_i(a/n)\phi_j(a/n)\IE\bbbclc{\psi_i(U_a) \psi_j(U_a)\prod_{v\overset{F_i}{\sim}w} \Var\bclr{Y_{a_{ij}}\given U}}\kern-7em \\[-1ex]
      &\text{if $F_i =  F_j$,}\\[2ex]
0& \text{otherwise}
\end{cases}
}
(we emphasise that in \eq{27}, the condition `$F_i =  F_j$' really means equality including vertex labels, not just that $F_i$ and $F_j$ are isomorphic).
Before stating our main result, we need some more notation. For a multi-index $\alpha = (\alpha_1, \dots, \alpha_d)$ of non-negative integers and $z\in\IR^d$, let
\be{
  \abs{\alpha}=\alpha_1 + \cdots + \alpha_d,
  \qquad
  \alpha! = \alpha_1!\cdots\alpha_d!,
  \qquad
  z^\alpha = z_1^{\alpha_1}\cdots z_d^{\alpha_d}
}
and 
\be{
  \partial^{\alpha} g(x) =\frac{ \partial^{\abs{\alpha} } g(x)}{\partial x_1^{\alpha_1} \cdots \partial x_d^{\alpha_d} } .
}
For any multi-index $\alpha\in\IN^d$, let 
\be{
  \abs{h}_\alpha = \sup_{x\in\IR^d}\abs{\partial^\alpha h(x)}.
}
Moreover, for two $d$-dimensional random vectors $X$ and $Y$, and with $\cK$ the class of convex sets in $\IR^d$, define the convex set distance
\be{
  d_c\bclr{\law(X),\law(Y)} =  \sup_{A\in \cK} \abs{\IP[X\in A] - \IP[Y\in A] }.
}

\begin{theorem} \label{thm1}
Let $W$ be defined as in \eqref{26}, and let $Z=( Z_1,\dots,  Z_d)$ be a centred Gaussian random vector with covariance matrix $\Sigma$ as given by \eqref{27}, and assume the $Y_{ij}$, $\phi_i$ and $\psi_i$ are all bounded. Let $p$ be an odd integer such that $p\geq  \max \{k_1,\dots, k_d\}$. Then, for any $(p+2)$-times partially differentiable function $h:\IR^d\to\IR$
\ben{
  \abs{\IE h(W)-\IE h(Z)} \leq \frac{C\sup_{\alpha:\abs{\alpha}\leq p+2}\abs{h}_\alpha}{ n^{1/2}}
\label{28}
}
for some constant $C$ that is independent of $n$. Moreover, 
\ben{\label{29}
  d_c\bclr{\law(W_n),\law(Z)} \leq Cn^{-\frac{1}{2(p+2)}}
}
again for some constant $C$ that is independent of $n$. 
\end{theorem}

\begin{remark} Consider again the example of the $2\times2$ graphon and $t_\triangle^{\inj}(G_n)$ from Section~\ref{sec1-1}. For example, letting $W_1$ equal $W$ from \eq{1e}, we can write 
\be{
  W_1 = n^{-1/2}\sum_{a=1}^n \phi_1(a/n)\bclr{\psi_1( U_{a})-\IE\psi_1( U_{a})},
  \qquad \phi_1\equiv 1,\, \psi_1(u)=u;
}
letting $W_2$ equal $V_{\edge,1}$ from \eq{1e}, we can write
\be{
  W_2 = \sbinom{n}{2}^{-1/2}  \sum_{a\in \cI^n_{2}} \phi_2(a/n)\,\psi_2(U_a) \prod_{v\overset{\kern0.0em\edge}{\sim}w}\clr{Y_{a_va_w}-\kappa(U_{a_v},U_{a_w})},
  \qquad \phi_2\equiv1,\,\psi_2\equiv 1;
}
letting $W_3$ equal $V_{\edge,4}$ from \eq{1b}, we can write
\bm{
  W_3 = \sbinom{n}{2}^{-1/2}  \sum_{a\in \cI^n_{2}} \phi_3(a/n)\,\psi_3(U_a) \prod_{v\overset{\kern0.0em\edge}{\sim}w}\clr{Y_{a_va_w}-\kappa(U_{a_v},U_{a_w})},
  \\[-1ex] \phi_3\equiv1,\, \psi_3(u_1,u_2) = u_1u_2;
}
finally, letting $W_4$ equal $V_{\twostar,1}$ from \eq{1b}, we can write
\bm{
  W_4 = \sbinom{n}{3}^{-1/2}  \sum_{a\in \cI^n_{3}} \phi_4(a/n)\,\psi_4(U_a) \prod_{v\overset{\kern-0.0em\twostar}{\sim}w}\clr{Y_{a_va_w}-\kappa(U_{a_v},U_{a_w})},
  \\[-1ex] \phi_4\equiv1,\, \psi_4(u_1,u_2,u_3) = \kappa(u_1,u_3).
}
The other quantities in \eq{1b} can be represented in a similar manner. Applying Theorem~\ref{thm1}, we obtain that the quantities in \eq{1e} and \eq{1b} are jointly close in distribution to independent Gaussian random variables.  
\end{remark}

\begin{remark}\label{rem1}Consider the case of $\IG(n,\kappa)$; that is, the $U_i$ are independent and distributed uniformly on $[0,1]$. Assume $\phi_i$ and $\psi_i$ are continuous almost everywhere; then, with $\bar\psi_i(u) = \psi_i(u)-\IE\psi_i(U_1)$ if $k_i=1$,
\besn{\label{30}
\lim_{n\toinf}\sigma_{ij}= 
  \begin{cases} 
    \displaystyle \int_{[0,1]} \phi_i(t)\phi_j(t)\,dt\,\int_{[0,1]}\-\psi_i(u)\-\psi_j(u)du
      &\text{if $k_i=k_j=1$,}\\[5ex]
     \displaystyle\int_{\cD_k} \phi_i(t)\phi_j(t)\,dt\,\int_{[0,1]^k}{\psi_i(u) \psi_j(u)\prod_{v\overset{F_i}{\sim}w} \kappa(u_v,u_w) \clr{1-\kappa(u_v,u_w)}}du\kern-7em \\[-2ex]
      &\text{if $F_i =  F_j$,}\\[2ex]
0& \text{otherwise.}
\end{cases}
}
Therefore, the corresponding Gaussian stochastic measures $Z_F$ are determined by the measure spaces 
\be{
  \bbbclr{\cD_k\times[0,1]^k,dt\times \prod_{v\overset{F}{\sim}w} \kappa(u_v,u_w) \clr{1-\kappa(u_v,u_w)}\, du}, \qquad F\in \cF,\, k=\abs{F},
}
and these measures are independent of each other.
\end{remark}

\begin{remark}\label{rem2} Now consider $\IG_{\mathrm{lat}}(n,\kappa)$. The case $k_i=1$ is not interesting since $\sigma_{ij}=0$ for all $j$. Moreover, we can assume without loss of generality that $\psi_i\equiv 1$. Assume that $\phi$ and $\kappa$ are continuous almost everywhere. Then, if $F_i=F_j$, we have
\besn{\label{31}
\lim_{n\toinf}\sigma_{ij} =
  \int_{\cD_k} \phi_i(t)\phi_j(t)\prod_{v\overset{F}{\sim}w} \kappa(t_v,t_w) \clr{1-\kappa(t_v,t_w)}\,dt
}
and $\sigma_{ij}=0$ otherwise. Therefore, the corresponding Gaussian stochastic measure $Z_F$ is determined by the measure space
\be{
  \bbbclr{\cD_k,\prod_{v\overset{F}{\sim}w} \kappa(t_v,t_w) \clr{1-\kappa(t_v,t_w)} \, dt},
  \qquad F\in \cF,\,k=\abs{F},
}
and these measures are independent of each other.
\end{remark}

\subsection{Connection to fourth moment theorem} 

One might wonder why the limits of the centred subgraph count statistics for connected $F$ turn out to be Gaussian. We believe that this is connected to the \emph{Fourth Moment Theorem}, first proved by \cite{Nualart2005}; see \cite{Nourdin2012} for a comprehensive discussion and proofs based on Stein's method. 

The theorem can be formulated as follows. Let $(Z_h)_{h\in H}$ be a Gaussian Hilbert space defined on some probability space $\Omega$, and let $\cF$ be the sigma-algebra generated by that space. Let $F_n\in L_2(\Omega,\cF)$ with $\IE F_n=0$ and $\Var F_n = 1$ for all $n\geq 1$, and assume the $F_n$ are elements of a fixed Wiener chaos. Then, $F_n$ converges to a standard Gaussian distribution if and only if $\IE F^4_n$ converges to $3$.

Consider a Gaussian stochastic measure $Z_2$ on $\cD_2$, where $\cD_k$ is as before --- we can think of $Z_2$ as the Gaussian approximation of the centred and scaled ``edge-field'' $\^Z_n$ in \eq{4}. For $1\leq i<j \leq n$, let 
\be{
  X_{ij} = \int_{\cls{\frac{i-1}{n},\frac{i}{n}}\times\cls{\frac{j-1}{n},\frac{j}{n}}} Z_2(dx,dy).
}
It is easy to see that the $X_{ij}$ are independent and that  
\be{
  X_{ij}\sim \N(0,n^{-2}).
}
We can think of $X_{ij}$ as a Gaussian version of the centred and scaled edge indicator between vertices $i$ and $j$, where $i<j$.   

Let $\phi\in L_2(\cD_3)$ and assume $\phi$ is continuous almost everywhere, and define $\-\phi_n\in L_2(\cD_2\times\cD_2 )$ as 
\be{
  \-\phi_n(x,y,u,v)=\phi\bbclr{\frac{\ceil{nx}}{n},\frac{\ceil{ny}}{n},\frac{\ceil{nv}}{n}} \I\bcls{\ceil{ny}/{n}=\ceil{nu}/{n}}.
} 
Note that 
\besn{\label{100}
 &\int_{\cD_2}\int_{\cD_2} \-\phi_n(x,y,u,v)\, Z_2(dx,dy)Z_2(du,dv)\\
 &\kern6em=\sum_{i<j<k} \phi\bbclr{\frac{i}{n},\frac{j}{n},\frac{k}{n}}  X_{ij} X_{jk}
 + \text{small boundary term}
}
lives in the second Wiener chaos and has variance
\be{
  \frac{1}{n^4}\sum_{i<j<k}
  \phi\bbclr{\frac{i}{n},\frac{j}{n},\frac{k}{n}}^2 
  \approx \frac{1}{n}\int_{\cD_3} \phi(x)^2dx.
}
The sum on the right hand side of \eq{100} is just the two-star count of the $X_{ij}$, and so centred subgraph counts of random graphs are in essence multiple stochastic integrals of the centred edge indicators. 

Now, assume that $\int_{\cD_3} \phi(x)^2dx=1$ and consider
\be{
 F_n = \sum_{i<j<k} \sqrt{n}\phi\bbclr{\frac{i}{n},\frac{j}{n},\frac{k}{n}}  X_{ij} X_{jk};
}
we have $\Var F_n = 1$. It is not difficult to see that, if $\phi$ is continuous almost everywhere,
\bes{
  \IE F_n^4 & = \sum_{i<j<k} n^2 \phi\bbclr{\frac{i}{n},\frac{j}{n},\frac{k}{n}}^4\times \frac{3}{n^8} \\
  &\qquad+ 3\sum_{i<j<k}\sum_{\substack{u<v<w\\(u,v,w)\neq(i,j,k)}}
  n^2 \phi\bbclr{\frac{i}{n},\frac{j}{n},\frac{k}{n}}^2
  \phi\bbclr{\frac{u}{n},\frac{v}{n},\frac{w}{n}}^2
  \times \frac{1}{n^8} \\
 & = 3\bbbclr{\int_{\cD_3}\phi(x)^2dx}^2 + \lito(1)
 = 3 + \lito(1),
}
and so, by the fourth moment theorem, $F_n$ converges to a standard normal. The corresponding multivariate convergence can be made with similar arguments for any finite collection of such $\phi$. Formally, we can therefore identify a Gaussian Hilbert Space on $L_2(\cD_3)$ with 
\ben{\label{32}
  \int_{\cD_3} \phi(x_1,x_2,x_3)\,Z_2(dx_1,dx_2) Z_2(dx_2,dx_3),
  \qquad \phi\in L_2(\cD_3).
}
This argument is general, and in the same manner, we can think of $Z_2$ giving rise to a Gaussian Hilbert space on $L_2(\cD_k)$ for every connected graph $F$ on $k$ vertices and identify it with the integral
\ben{\label{101}
  \int_{\cD_k} \phi(x)\, \prod_{i\stackrel{F}{\sim} j}  Z_2(dx_i,dx_j),
  \qquad \phi\in L_2(\cD_k).
} 
However, it is important to keep in mind that the convergence of $F_n$ is only distributional, so it is not clear whether the Gaussian Hilbert spaces \eq{101} can be coupled with the underlying space $Z_2$ in a non-trivial and meaningful manner.

\section{Abstract approximation theorem} 

The following abstract multivariate normal approximation theorem is based on Stein's method and can yield informative bounds even in the case of vectors of sums of uncorrelated, but not necessarily independent random variables. Let $e_i$ be the $i$-th unit-vector in $\IN^d$. A triple of $d$-dimensional vectors $(W,W',G)$ is called \emph{Stein coupling} if
\ben{\label{33}
 \IE\clc{ G^t g(W') - G^t g(W)}=\IE \clc{W^t g(W)}.
}
for all $g:\IR^d\to \IR^d$ for which the expectations exists; see \cite{Chen2010b} and \cite{Fang2012a}.
 
\begin{theorem} \label{thm2} Let $(W,W',G)$ be a Stein coupling, and let $\Sigma=(\sigma_{ij})_{1\leq i,j\leq d}=\Cov(W)$; set $D=W'-W$. 
Then, for any  $(p+2)$-times  differentiable function $h:\IR^d\to\IR$ and with $Z$ being a centred Gaussian vector with covariance matrix $\Sigma$,
\besn{\label{34}
& \abs{ \IE h(W) - \IE h(Z)} \\ 
&\qquad\leq  \sum_{i=1}^d \bbbbclr{\, \sum_{\alpha: 1\leq\abs{\alpha}\leq p} \frac{\abs{h}_{\alpha+e_i}}{(\abs{\alpha}+1)\alpha!} \sqrt{\Var\IE\clc{G_iD^{\alpha}\given W}}
\\
&\qquad\quad+  \sum_{\alpha: 2\leq \abs{\alpha}\leq p} \frac{\abs{h}_{\alpha+e_i}}{(\abs{\alpha}+1)\alpha!} \abs{\IE \clc{G_iD^{\alpha}}}
 +  \sum_{\alpha: \abs{\alpha}=p+1} \frac{\abs{h}_{\alpha+e_i}}{(p+2)\alpha!} \IE\abs{G_iD^\alpha}}.
}
\end{theorem}

\begin{remark}\label{rem3}
Note that, by Young's inequality, we can upper bound the last term in~\eqref{34} as
\ben{\label{35}
  \IE\abs{G_iD^\alpha} \leq \norm{G_i}_\infty\sum_{j=1}^d\frac{\alpha_j}{p+1} \IE\babs{D_j^{p+1}}.
} 
\end{remark}
\begin{proof}[Proof of Theorem \ref{thm2}]
Let $g:\IR^d \to \IR$ be a solution to the Stein's equation 
\ben{\label{36}
 \sum_{i,j=1}^d \sigma_{ij}\partial_{ij} g(z) - \sum_{i=1}^dz_i\partial_ig(z) =  h(z) - \IE h(Z),\qquad z\in\IR^d.
}
From \cite[Eq.~(10)]{Meckes2009}, it is immediate that
\ben{\label{37}
  \abs{g}_\alpha\leq\frac{1}{\abs{\alpha}} \abs{h}_\alpha
}
and it is therefore enough to bound
\be{
  \IE \bbbclc{ \sum_{i,j=1}^d \sigma_{ij}\partial_{ij} g(W) - \sum_{i=1}^dW_i\partial_ig(W)}
}
in order to bound the left hand side of \eqref{34}. By Taylor's theorem for multivariate functions, 
\ba{
f(w')- f (w) & = \sum_{l=1}^p  \sum_{\alpha: \abs{\alpha}=l} \frac{\partial^\alpha f(w)}{\alpha!} (w'-w)^{\alpha} + R^{(p+1)}(w',w)
}
and 
\ba{
R^{(p+1)} (w',w) & = \sum_{\alpha: \abs{\alpha}=p+1} \frac{p+1}{\alpha!} (w'-w)^\alpha \int_0^1  (1-s)^{p} \partial^{\alpha} f (w+s(w'-w)) ds.
}
Now, using \eqref{33}, we have 
\bes{
& \IE \bbbclc{ \sum_{i=1}^dW_i\partial_ig(W)} \\
&\qquad = \IE \bbbclc{  \sum_{i=1}^dG_i\bclr{\partial_ig(W')-\partial_ig(W)}}\\
&\qquad = \IE \bbbclc{  \sum_{i=1}^dG_i\sum_{l=1}^p  \sum_{\alpha: \abs{\alpha}=l} \frac{\partial^{\alpha+e_i} g(W)}{\alpha!} D^{\alpha}}\\
&\qquad\quad+\IE \bbbclc{  \sum_{i=1}^dG_i\sum_{\alpha: \abs{\alpha}=p+1} \frac{p+1}{\alpha!} D^\alpha \int_0^1  (1-s)^{p} \partial^{\alpha+e_i} g (W+sD) ds} \eqqcolon r_1+r_2.
}
Now, 
\bes{
  r_1 & = \sum_{l=1}^p \IE \bbbclc{ \sum_{i=1}^d  \sum_{\alpha: \abs{\alpha}=l} \frac{\partial^{\alpha+e_i} g(W)}{\alpha!} G_iD^{\alpha}}\\
  & = \sum_{l=1}^p \IE \bbbclc{ \sum_{i=1}^d  \sum_{\alpha: \abs{\alpha}=l} \frac{\partial^{\alpha+e_i} g(W)}{\alpha!} \bclr{\IE\clc{G_iD^{\alpha}\given W}-\IE \clc{G_iD^{\alpha}}}}\\
  &\quad+\sum_{l=1}^p \IE \bbbclc{ \sum_{i=1}^d  \sum_{\alpha: \abs{\alpha}=l} \frac{\partial^{\alpha+e_i} g(W)}{\alpha!} \IE \clc{G_iD^{\alpha}}} \eqqcolon r_{1,1}+r_{1,2}.
}
First, 
\be{
  \abs{r_{1,1}}\leq \sum_{l=1}^p \sum_{i=1}^d  \sum_{\alpha: \abs{\alpha}=l} \frac{\abs{g}_{\alpha+e_i}}{\alpha!} \IE\babs{\IE\clc{G_iD^{\alpha}\given W}-\IE \clc{G_iD^{\alpha}}}.
}
Next, recalling that $\sigma_{ij}=\IE\clc{G_i D_j}$,
\be{
  \bbbabs{r_{1,2}-\IE\sum_{i,j=1}^d\sigma_{ij}\partial_{ij}g(W)}
  \leq
  \sum_{l=2}^p \sum_{i=1}^d  \sum_{\alpha: \abs{\alpha}=l} \frac{\abs{g}_{\alpha+e_i}}{\alpha!} \abs{\IE \clc{G_iD^{\alpha}}},
}
A bound on $r_2$ can be obtained in a similar manner, and so 
\besn{\label{38}
& \bbbabs{ \IE \bbbclc{ \sum_{i,j=1}^d \sigma_{ij}\partial_{ij} g(W) - \sum_{i=1}^dW_i\partial_ig(W)}} \\ 
&\qquad\leq  \sum_{i=1}^d \bbbbclr{\, \sum_{\alpha: 1\leq\abs{\alpha}\leq p} \frac{\abs{g}_{\alpha+e_i}}{\alpha!} \sqrt{\Var\IE\clc{G_iD^{\alpha}\given W}}
\,{}+  \sum_{\alpha: 2\leq \abs{\alpha}\leq p} \frac{\abs{g}_{\alpha+e_i}}{\alpha!} \abs{\IE \clc{G_iD^{\alpha}}}\\
&\qquad\kern4em
 +  \sum_{\alpha: \abs{\alpha}=p+1} \frac{\abs{g}_{\alpha+e_i}}{\alpha!} \IE\abs{G_iD^\alpha}}.
}
Applying \eqref{37}, the claim follows.
\end{proof}

\section{Proof of Theorem~\ref{thm1}}

Fix $d\geq 1$, and for each $1\leq i \leq d$, let $F_i$ be a connected graph on the vertex set $[k_i]$. Let $U=(U_v)_{1\leq v\leq n}$ be independent random variables, let $\kappa$ be a graphon, and let $Y=(Y_{vw})_{1\leq v<w\leq n}$ be random variables that are independent conditionally on $U$ and such that $\IE\clc{Y_{vw}\given U_v,U_w}=\kappa(U_v,U_w)$. For $1\leq i\leq d$ and $a\in\cI^n_{k_i}$, let
\ben{\label{39}
  T_{i,a} = \prod_{v\stackrel{F_i}{\sim} w}\bclr{Y_{a_v a_w}-\kappa(U_{a_v},U_{a_w}) }
}
if $k_i\geq 2$, and for convenience, set $T_{i,a} = 1$ if $k_i=1$.
For each $1\leq i\leq d$, let $\psi_i:[0,1]^{k_i}\to\IR$ be a bounded function, and for $a\in\cI^n_{k_i}$, let 
\be{
  \Phi_{i,a} = 
  \begin{cases}
    \psi_i(U_{a_1})-\IE  \psi_i(U_{a_1}) & \text{if $k_i=1 $,}\\
    \psi_i(U_a) & \text{if $k_i\geq 2$.}
  \end{cases}
} 
Now, for $a\in \cI^n_{k_i}$, let

\be{
  X_{i,a} = {\textstyle\binom{n}{ k_i}^{-1/2}}\phi_i(a/n)\Phi_{i,a} T_{i,a}.
} 
Recalling the definition of $W=(W_1,\dots,W_d)$ from \eqref{26}, we have 
\be{ 
W_i = \sum _{a\in \cI^n_{k_i}  } X_{i,a} .
}

\subsection{Stein Coupling}

Let $1\leq i\leq d$ and $a\in\cI^n_{k_i}$. For $1\leq j\leq d$, let
\be{
  N^{i,a}_j = 
    \{b\in\cI^n_{k_j}:\abs{a\cap b}\geq 2\wedge k_i\},
}
where in the expression $a\cap b$, the ordered tuples $a$ and $b$ are interpreted as unordered sets. Note that if $k_i\geq 2$ and $k_j=1$ then $N^{i,a}_j=\emptyset$. Let
\be{
  W^{i,a}_j = W_j - \sum_{b\in N^{i,a}_j} X_{i,b}.
}
Let $I$ be uniformly distributed on $[d]$ and independent of all else, and given $I$, let $A$ be uniformly distributed on $\cI^n_{k_I}$. Let
\be{
  W' = W^{I,A} = \bclr{W^{I,A}_1,\dots,W^{I,A}_d},\qquad
  G = -d \binom{n}{ k_I} X_{I,A}\,e_I,
}
where $e_i$ is the $i$-th unit vector in $\IR^d$.

\begin{lemma}\label{lem3}
 $(W, W',G)$ is a $d$-dimensional Stein coupling.
\end{lemma}
\begin{proof} Write $g(x)=(g_1(x),\dots,g_d(x))$; averaging over $I$ and $A$, 
\besn{
  \IE\clc{G^t g(W')} & = -\sum_{i=1}^d\sum_{a\in\cI^n_{k_i}}{\textstyle \binom{n}{ k_i}^{-1/2}} \phi_i(a/n)\IE\clc{ \Phi_{i,a} T_{i,a}g_i(W^{i,a}) }.
}

If $k_i=1$, then $W^{i,a}$ does not contain any information about $U_a$, and since $T_{i,a}=1$ and $\IE\Phi_{i,a}=0$, it follows that $\IE\clc{ \Phi_{i,a} T_{i,a}g_i(W^{i,a}) }=0$. If $k_i\geq 2$, then conditionally on $U$, $W^{i,a}$ does not contain any information about $(Y_{vw})_{v,w\in a}$. Since $\IE\clc{T_{i,a}\given U}=0$ it again follows that $\IE\clc{ \Phi_{i,a} T_{i,a} g_i(W^{i,a}) }=0$. Hence $ \IE\clc{G^t g(W')}=0$. It is straightforward to check that $-\IE\clc{G^t g(W) } = \IE\clc{W^t g(W)}$.
\end{proof}

\subsection{Estimates on mixed moments}

Before proving the main theorem, we present some lemmas, which will be used in the proof of Theorem~\ref{thm1}. For a graph $F$ on the vertex set $[k]$ and $b\in\cI^n_{k}$, denote by $F(b)$ the graph on the vertex set $\{b_1,\dots,b_k\}$ where $b_v$ and $b_w$ are connected in $F(b)$ if and only if $v$ and $w$ are connected in $F$. In other words, $F(b)$ is the induced graph when mapping vertex $v$ to vertex $b_v$ for all $v\in[k]$. We assume throughout that, for each $1\leq i\leq d$, $F_i$ is a connected graph on the vertex set $[k_i]$, where $k_i\geq 1$. The reader should keep in mind that the bounds obtained in Lemmas~\ref{lem5}--\ref{lem7} are \emph{worst-case} bounds, and will typically be sharp if all graphs involved are line graphs, but depending on the combinatorics of the $F_i$, the bounds could be much smaller. Phrases like ``there are $\bigo(n^k)$ choices'' have to be understood in the context of the usual Bachmann–Landau notation, which in this case means that the number of choices can be ``of order $n^k$ or of \emph{smaller order}''.

\begin{lemma}[{c.f. \cite[Lemma~5]{Janson1991}}]\label{lem4} Let $m\geq 2$, and for each $1\leq l\leq m$, let $1\leq i_l\leq d$, and let $b_l\in\cI^n_{k_{i_l}}$. Assume 
\ben{\label{40}
  \IE \bbbclc{\prod_{l=1}^m \Phi_{i_l,b_l}\prod_{l=1}^m T_{i_l,b_l}} \neq 0.
}
Then, every vertex and every edge belong to at least two of the subgraphs 
\ben{\label{41}
  F_{i_1}(b_1),\cdots,F_{i_m}(b_m).
} 
Moreover, the subgraphs \eqref{41} either coincide in $m/2$ disjoint pairs ($m$ necessarily even) or there is a vertex that belongs to at least three of them.
\end{lemma}

\begin{proof} Without loss of generality, assume there is $m'\leq m$ such that $k_{i_l}=1$ for all $l> m'$ (if there are no such indices, set $m'=m$). So, assume \be{\IE \bbbclc{\prod_{l=1}^m \Phi_{i_l,b_l}\IE\bbbclc{\prod_{l=1}^{m'} T_{i_l,b_l}\given U}} \neq 0.}
Suppose there is an edge between $v$ and $w$ in a subgraph that is not in any other subgraph, so that the factor $Y_{vw}-\kappa(U_v,U_w)$ appears exactly once in $\prod_{l=1}^{m'} T_{i_l,b_l}$. Since the~$Y_{vw}$ are conditionally independent given $U$ and since $\IE\clc{Y_{vw}-\kappa(U_v,U_w)\given U}=0$, it would follow that $\IE \bclc{\prod_{l=1}^{m'} T_{i_l,b_l} \given U}= 0$, which contradicts the claim. Also, as a consequence, every vertex among the subgraphs that has at least one edge attached to it, must also appear in another subgraph. Suppose now there is an isolated vertex $v$ in a subgraph, say $F_{i_l}(b_l)$ for some $l>m'$, that is not in any other subgraph. In that case, $U_v$ only appears in~$\Phi_{i_l,b_l}$ and $\IE \bclc{\prod_{l=1}^{m'} T_{i_l,b_l} \given U}$ does not depend on $U_v$. Due to the fact that $\IE\Phi_{i_l,b_l}=0$ for such~$F_{i_l}(b_l)$ and independence, the left hand side of \eqref{40} would equal zero, again in contradiction to the claim. This concludes the proof of the first assertion.

To prove the second assertion, assume each vertex appears in exactly two of the $F_{i_l}(b_l)$. If a vertex is in $F_{i_l}(b_l)$ and $F_{i_{l'}}(b_{l'})$, say, then all edges attached to it, must also be in $F_{i_l}(b_l)$ and $F_{i_{l'}}(b_{l'})$, and so forth. Since both graphs are connected, they must coincide. Hence, the $F_{i_l}(b_l)$ must come in identical pairs.  
\end{proof}

\begin{lemma}\label{lem5} Let $1\leq i,i_1,\dots,i_m\leq d$ for some $m\geq 2$. Then there exists a constant $C>0$ that is independent of $n$ such that 
\ben{\label{42}
  \bbbbabs{\IE\sum_{a\in\cI^n_{k_{i}}}\sum_{b_1\in N^{i,a}_{i_1}}\cdots\sum_{b_m\in N^{i,a}_{i_m}}
  X_{i,a}X_{i_1,b_1}\cdots X_{i_m,b_m}}\leq Cn^{-(m-1)/2}.
}
\end{lemma}
\begin{proof} First, write the expectation on the left hand side of \eqref{42} as
\besn{\label{43}
 \xi
 & \coloneqq 
 \frac{1}{\binom{n}{k_{i}}^{1/2}\times \binom{n}{k_{i_1}}^{1/2}\times\cdots\times\binom{n}{k_{i_m}}^{1/2}} \\
 &\quad  \times \sum_{a\in\cI^n_{k_{i}}}\sum_{b_1\in N^{i,a}_{i_1}}\cdots\sum_{b_m\in N^{i,a}_{i_m}}
      \IE\bclc{\Phi_{i,a} T_{i,a}\Phi_{i_1,b_1} T_{i_1,b_1}\cdots\Phi_{i_m,b_m} T_{i_m,b_m}}.
}
Fix $a,b_1,\dots,b_m$ and consider the induced subgraphs
\ben{\label{44}
  F_{i}(a),F_{i_1}(b_1),\dots,F_{i_m}(b_m),
}
which are subgraphs on the vertex set $[n]$. By Lemma~\ref{lem4}, if the corresponding expectation of the summand in \eqref{43} is non-zero, then either these subgraphs coincide in pairs of disjoint subgraphs, or all vertices and edges appear in at least two subgraphs while at least one vertex appears in three. Note that all the subgraphs share vertices with $F_i(a)$ by the definition of $N^{i,a}_j$, and thus can coincide in distinct pairs only if $m=1$, which is excluded. 

Assume $k_i\geq 2$, and recall that each of the $F_{i_l}(b_l)$ shares at least two vertices with $F_i(a)$. Also, note that $F_i(a)$ has $k_i$ vertices and so $\sum_{l=1}^m k_{i_l}$ must be at least $k_i$ in order for every vertex in $F_i(a)$ to also be in one of the other subgraphs. However, if $m\geq 2$, $\sum_{l=1}^m k_{i_l}$ must be larger than $k_i$ to also cover all \emph{edges} of $F_i(a)$, of which there are at least $k_i-1$; indeed, if a vertex of $F_i(a)$ has two edges attached to it and the two edges are contained in different subgraphs, say one in $F_{i_1}(b_1)$ and the other in $F_{i_2}(b_2)$, then that vertex must belong to all three subgraphs. Therefore, if $\sum_{l=1}^m k_{i_l}< k_i+m-1$, it is not possible that each edge of $F_i(a)$ also belongs to one of the other subgraphs, and so all terms in \eqref{43} vanish, that is, $\xi=0$, and the claim is trivially true.

If $\sum_{l=1}^m k_{i_l}=k_i+m-1$, the sum \eqref{43} contains at most $\bigo(n^{k_i})$ non-zero terms, since all vertices of $F_{i_1}(b_1),\dots,F_{i_m}(b_m)$ must coincide with vertices of $F_i(a)$ to cover all of the latter, and this arrangement contributes only a combinatorial factor to the sum that is independent of $n$. Thus,
\be{
  \abs{\xi}\leq\frac{Cn^{k_i}}{n^{k_i/2}n^{(k_{i_1}+\cdots+k_{i_m})/2}} \leq 
  \frac{C}{n^{(m-1)/2}},
}
where the second inequality follows from the fact that $\sum_{l=1}^m k_{i_l}= k_i+m-1$.

If $\sum_{l=1}^m k_{i_l}>k_i+m-1$, let $q:=\sum_{l=1}^m k_{i_l}-(k_i+m-1)$. Note that $q$ is the maximal number of vertices available among $F_{i_1}(b_1),\dots,F_{i_m}(b_m)$ that do not need to overlap with $F_i(a)$ (there might be fewer that can be chosen outside of $F_i(a)$, but in any case, never \emph{more}). Assume first $q$ is even. Since every vertex must be contained in at least two subgraphs, there are $q/2$ additional free choices in \eqref{43}, contributing a factor of $\bigo(n^{q/2})$ to the sum, so that  
\be{
  \abs{\xi}\leq\frac{Cn^{k_i+q/2}}{n^{k_i/2}n^{(k_{i_1}+\cdots+k_{i_m})/2}} \leq 
  \frac{C}{n^{(m-1)/2}},
}
where the second inequality follows from the fact that $\sum_{l=1}^m k_{i_l}= k_i+q+m-1$. If $q$ is odd, one of the $q$ vertices cannot be chosen freely, so that the additional factor appearing in $\bigo(n^{(q-1)/2})$, and we obtain
\be{
  \abs{\xi}\leq\frac{Cn^{k_i+(q-1)/2}}{n^{k_i/2}n^{(k_{i_1}+\cdots+k_{i_m})/2}} \leq 
  \frac{C}{n^{(m-1)/2+1/2}}\leq \frac{C}{n^{(m-1)/2}}.
}

If $k_i=1$, coverage of $F_i(a)$ is always guaranteed, since every $F_{i_l}(b_l)$ overlaps with the one vertex of $F_i(a)$. Hence, with $q=\sum_{l=1}^m k_{i_l}-m$, there are at most $q/2$ vertices which can be  chosen freely if $q$ is even and $(q-1)/2$ if $q$ is odd, contributing a factor of no more than $\bigo(n^{q/2})$ to \eqref{43}, so that 
\be{
  \abs{\xi}\leq\frac{Cn^{1+q/2}}{n^{1/2}n^{(k_{i_1}+\cdots+k_{i_m})/2}} \leq 
  \frac{Cn^{1/2}}{n^{m/2}}.
}
This concludes the proof.
\end{proof}

\begin{lemma}\label{lem6} Let $1\leq i,j\leq d$ and let $m\geq 2$. Then there exists a constant $C>0$ that is independent of $n$ such that 
\ben{\label{45}
   \bbbbabs{ \IE\sum_{a\in\cI^n_{k_{i}}}\sum_{b_1\in N^{i,a}_{j}}\cdots\sum_{b_m\in N^{i,a}_{j}}
  X_{j,b_1}\cdots X_{j,b_m} } \leq \begin{cases}
  Cn^{1-m/2} & \text{if $k_i=1$,}\\
  Cn^{k_i-m} & \text{if $k_i\geq 2$.}
  \end{cases}
}
\end{lemma}

\begin{proof} First, write
\besn{\label{46}
 \xi
 & \coloneqq 
 \frac{1}{\binom{n}{k_{j}}^{m/2}} \sum_{a\in\cI^n_{k_{i}}}\sum_{b_1\in N^{i,a}_{j}}\cdots\sum_{b_m\in N^{i,a}_{j}}
      \IE\bclc{\Phi_{j,b_1} T_{j,b_1}\cdots\Phi_{j,b_m} T_{j,b_m}}
}
Fix $a,b_1,\dots,b_m$ and consider the induced subgraphs
\ben{\label{47}
  F_j(b_1),\dots,F_j(b_m),
}
which are subgraphs on the set $[n]$. By Lemma~\ref{lem4}, if the corresponding summand in \eqref{46} is non-zero, every vertex must appear in at least two of these subgraphs. 

Assume $k_i\geq 2$, and $k_j\geq 2$ (if $k_j=1$, then $N^{i,a}_j=\emptyset$ and the claim is trivially true). There are $\bigo(n^{k_i})$ choices for $a$ and since each of the $F_j(b_l)$ must have two vertices in the set $a$, we can assign $k_j-2$ vertices freely for each such subgraph, subject to the condition that each vertex appears twice. With $q=m(k_j-2)$, there are $\bigo(n^{q/2})$ choices if $q$ is even. Hence
\ben{\label{48}
  \abs{\xi}\leq\frac{Cn^{k_i+q/2}}{n^{mk_j/2}}\leq Cn^{k_i-m}.
}
If $q$ is odd, there are $\bigo(n^{(q-1)/2})$ choices, hence
\ben{\label{49}
  \abs{\xi}\leq\frac{Cn^{k_i+(q-1)/2}}{n^{mk_j/2}}\leq Cn^{k_i-m-1/2} \leq  Cn^{k_i-m}.
}
In the case $k_i=1$, similar arguments lead to the estimate 
\ben{\label{50}
  \abs{\xi}\leq\frac{Cn^{1+m(k_j-1)/2}}{n^{mk_j/2}}\leq Cn^{1-m/2},
}
if $m(k_j-1)$ is even, and similarly if it is odd. This concludes the proof.
\end{proof}

\begin{lemma}\label{lem7} Let $1\leq i,i_1,\dots,i_m\leq d$ for some $m\geq 1$. Then there exists a constant $C>0$ that is independent of $n$ such that 
\bmn{\label{51}
  \bbbbabs{\sum_{a,a'\in\cI^n_{k_{i}}}\sum_{b_1\in N^{i,a}_{i_1}}\sum_{b_1'\in N^{i,a'}_{i_1}}\cdots\sum_{b_m\in N^{i,a}_{i_m}}\sum_{b_m'\in N^{i,a'}_{i_m}}\\
  \Cov\bclr{X_{i,a}X_{i_1,b_1}\cdots X_{i_m,b_m},X_{i,a'}X_{i_1,b_1'}X_{i_m,b_m'}}}
  \leq Cn^{-m}.
}
\end{lemma}
\begin{proof} First, let $\xi$ equal the left hand side of \eqref{51} without modulus. 
Consider first the case $k_i\geq 2$, in which case again we may assume $k_{i_l}\geq 2$ for all $1\leq l\leq m$, since otherwise $\xi=0$, using the same arguments as in the previous lemmas. Now, it is easy to verify that the covariances are zero if the two sets
\ben{\label{52}
  a\cup\bigcup_{l=1}^m b_l,\quad\text{and}\quad a'\cup\bigcup_{l=1}^m b_l'
}
do not overlap (independence). Fix $a,a',b_1,b_1',\dots,b_m,b_m'$, consider the induced subgraphs
\ben{\label{53}
  F_i(a),F_{i_1}(b_1),\dots,F_{i_m}(b_m),F_i(a'),F_{i_1}(b_1'),\dots,F_{i_m}(b_m'),
}
which are subgraphs on the set $[n]$, and also consider
\ben{\label{54}
  \IE\bclc{X_{i,a}X_{i_1,b_1}\cdots X_{i_m,b_m}\cdot X_{i,a'}X_{i_1,b_1'}\cdots X_{i_m,b_m'}}.
}
By Lemma~\ref{lem4}, if \eqref{54} is non-zero, every vertex in \eqref{53} must appear in at least two of these subgraphs. Now, let $r = \abs{a\cap a'}$.

\smallskip
\noindent\textit{Case $1\leq r \leq k_i$}: Since $r$ vertices in $F_i(a)$ are also in $F_i(a')$, both $F_i(a)$ and $F_i(a')$ have $k_i-r$ more vertices each that need to be in any of the other subgraphs and, since $F_i$ is connected, also at least $k_i-r$ more edges each. We proceed similarly as in the proof of Lemma~\ref{lem5}. If $\sum_{l=1}^m k_{i_l} < k_i - r +m$, it is not possible for all edges of $F_i(a\setminus a')$ and those connecting $F_i(a\setminus a')$ with $F_i(a\cap a')$, to be covered, and so $\xi=0$. Otherwise, let $q = \sum_{l=1}^m k_{i_l} - (k_i - r +m)$. There are $\bigo(n^{2k_i-r})$ choices for the vertices of $F_{i}(a)$ and $F_i(a')$ together, and there are $\bigo(n^{(2q)/2})$ choices for the remaining vertices. Hence,
\be{
  \abs{\xi} \leq \frac{Cn^{2k_i-r + q}}{n^{k_i+k_{i_1}+\cdots k_{i_m}}} \leq 
  \frac{Cn^{2k_i-r + q}}{n^{k_i+q+k_i-r+m}} \leq \frac{C}{n^{m}},
}
where we have used that $\sum_{l=1}^m k_{i_l}=q+k_i-r+m$

\smallskip
\noindent\textit{Case $r=0$}: $F_i(a)$ and $F_i(a')$ are not overlapping and each of $F_i(a)$ and $F_i(a')$ have $k_i$ vertices that need to be in any of the other subgraphs and, since $F_i$ is connected, also at least $k_i-1$ edges. If $\sum_{l=1}^m k_{i_l} < k_i - 1 +m$, it is not possible for all edges of $F_i(a)$ and $F_i(a')$, respectively, to be covered, and so $\xi=0$. Otherwise, let $q = \sum_{l=1}^m k_{i_l} - (k_i - 1 +m)$. There are $\bigo(n^{2k_i})$ choices for the vertices of $F_{i}(a)$ and $F_i(a')$ together, and there are $\bigo(n^{(2q)/2-1})$ choices for the remaining vertices, since at least one vertex from $\bigcup_{l=1}^m b_l$ must overlap with $\bigcup_{l=1}^m b_l'$. Hence,
\be{
  \abs{\xi} \leq \frac{Cn^{2k_i + q-1}}{n^{k_i+k_{i_1}+\cdots k_{i_m}}} \leq 
  \frac{Cn^{2k_i + q-1}}{n^{k_i+q+k_i-1+m}} \leq \frac{C}{n^{m}},
}
where we have used that $\sum_{l=1}^m k_{i_l}=q+k_i-1+m$.

Now suppose $k_1=1$. If $F_i(a)=F_i(a')$ there are $n$ choices for this one vertex, and since every subgraph must share a vertex with $F_i(a)$ and $F_i(a')$, respectively, there are $\bigo\bclr{n^{2\times\sum_{l=1}^m (k_{i_l}-1)/2}}$ choices for the remaining vertices. Hence,
\be{
  \abs{\xi}\leq \frac{Cn^{1+k_{i_1}+\cdots+k_{i_m}-m}}{n^{1+k_{i_1}+\cdots k_{i_m}}}
  \leq\frac{C}{n^{m}}.
} 
If $F_i(a)\neq F_i(a')$ there are $\bigo(n^2)$ choices for the two vertices, and since every subgraph must share a vertex with $F_i(a)$ and $F_i(a')$, respectively, there are $\bigo\bclr{n^{2\times\sum_{l=1}^m (k_{i_l}-1)/2-1}}$ choices for the remaining vertices, since at least one vertex from $\bigcup_{l=1}^m b_l$ must overlap with $\bigcup_{l=1}^m b_l'$. Hence
\be{
  \abs{\xi}\leq \frac{Cn^{2+k_{i_1}+\cdots+k_{i_m}-m-1}}{n^{1+k_{i_1}+\cdots k_{i_m}}}
  \leq\frac{C}{n^{m}}.
} 
This concludes the proof.
\end{proof}

\begin{lemma}\label{lem8}Let $m\geq 2$. Then
\bg{
  \abs{\IE D_j^m} \leq Cn^{-m},\qquad
  \abs{\IE\clc{G_i D^\alpha}} \leq Cn^{-(\abs{\alpha}-1)/2},\quad
  \Var\IE\clc{G_i D^\alpha\given U,Y}\leq Cn^{-\abs{\alpha}}.
}
\end{lemma}
\begin{proof}
Note that
\be{
  D_j = W_j'-W_j = -\sum_{b\in N^{I,A}_j}  X_{j,b},
}
and hence,
\be{
  \IE\clc{G_iD^\alpha\given U,Y} 
  = (-1)^{\abs{\alpha}+1}\sum_{a\in\cI^n_{k_i}}X_{i,a}
  \prod_{j=1}^d\bbbclr{\,\sum_{b\in N^{i,a}_j}  X_{j,b}}^{\alpha_j}
}
and
\be{
  \IE\clc{D_j^m\given U,Y} 
  = \frac{(-1)^{m}}{d}\sum_{i=1}^d{ \binom{n}{ k_i}^{-1}}\sum_{a\in\cI^n_{k_i}}\bbbclr{\,\sum_{b\in N^{i,a}_j}  X_{j,b} }^m.
}
The bounds are now a direct consequence of Lemmas~\ref{lem5}--\ref{lem7}.
\end{proof}

\subsection{Proof of Theorem~\ref{thm1}}

\begin{proof} The variance expressions \eqref{27} are straightforward to establish. The proof of the bounds \eqref{28} for $(p+2)$-times differentiable functions is a consequence of Theorem~\ref{thm2} and Remark~\ref{rem3} with the Stein coupling from Lemma~\ref{lem3}, along with the moment estimates of Lemma~\ref{lem8} with the choice $m=p+1$. 
 
We use the smoothing technique of \cite{Gan2017} in order to approximate the indicator function $I_{A}$ by a $(p+2)$-times partially differentiable function. Fix $A\in \cK$ and $\eps>0$,  define
\[A^\eps  = \{y\in \IR^d: d(y,A)<\eps\},\quad\text{and}\quad A^{-\eps } = \{y\in \IR^d: B(y;\eps) \subseteq A\}, \] 
where $d(y,A) = \inf_{x\in A} \abs{x-y}$ and $B(y;\epsilon)$ is the closed ball of radius $\eps$ around $y$.
Let $\{ h_{\eps,A}:\IR^d\to [0,1] ; A \in \cK\} $ be  a class of functions,  such that $h_{ \eps,A} (x)= 1 $ for $x\in A$ and $0$ for $x \notin A^\eps$. Then, by Lemma 2.1 of \cite{Bentkus2003}, we have for any $\eps>0$ that
\ben{\label{55}
\sup_{A\in \cK} \abs{\IP(W\in A) - \IP(Z\in A) } \leq 4 d^{1/4} \eps + \sup_{A\in \cK} \abs{\IE h_{\eps,A}(W)- \IE h_{\eps,A}(Z) } .
}
Let  $f:\IR^d\to \IR$ be a bounded and Lebesgue measurable function, and for $\delta>0$, consider the smoothing operator $S_\delta$ defined as
\be{
(S_\delta f)(x)  = \frac{1}{(2\delta)^d} \int_{x_1-\delta}^{x_1+\delta} \dots \int_{x_d-\delta}^{x_d+\delta}  f(z) dz_d \dots dz_1.
}
Choose $\delta  = \frac{\eps}{(p+3)^2\sqrt{d}}$, let $h_{\eps, A} = S_\delta^{p+3} I_{A^{\eps/(p+3)}}$; then  by Lemma 3.9 of \cite{Gan2017},  $h_{\eps, A} $ is $(p+2)$-times partially differentiable  and 
\be{
\norm{h_{\eps, A} }_\infty \leq 1,  \qquad     \abs{h_{\eps, A}}_\alpha  \leq \dfrac{1}{\eps^{\abs{\alpha}}}, \qquad 1\leq \abs{\alpha}\leq p+2.
}
Note that $h_{ \eps,A} (x)= 1 $ for $x\in A$ and $h_{ \eps,A} (x)= 0 $  for $x \notin A^\eps$. Therefore, from \eqref{28},
\ben{
 \abs{\IE h_{\eps,A} (W)-\IE h_{\eps,A}(Z)} 
   \leq \frac{C\sup_{\alpha:\abs{\alpha}\leq p+2}\abs{h_{\eps,A}}_\alpha}{ n^{1/2}} 
  \leq \frac{C}{ n^{1/2} \eps^{(p+2)}}
}
for some constant $C$. Now, using \eqref{55},  we have 
\besn{
\sup_{A\in \cK} \abs{\IP(W\in A) - \IP(Z\in A) }  \leq 4 d^{1/4} \eps + C   \dfrac{1}{\eps^{p+2} n^{1/2}}.
}
The final order $n^{-1/(2(p+2))} $ is then established by taking $\eps = n^{-1/(2(p+2))} $.
\end{proof}

\section{Proof of Lemmas~\ref{lem1} and~\ref{lem2}}

Consider the graph $F$ on the vertex set $[k]$ as fixed. In what follows, for any subgraph $H\subseteq F$,  $H^c$ denotes the `edge complement' of $H$ and is the graph obtained by removing from $F$ all edges which are present in $H$ and then removing all, if any, resulting isolated vertices.

\begin{lemma} \label{lem9} Recalling \eqref{23}, and with $F$ any graph on the vertex set $[k]$, we can write
\besn{\label{56}
\prod_{i\overset{F}{\sim} j} y_{ij}  = \sum_{H\subseteq' F} \rho_{F,H}(u,y),
}
where
\be{
  \rho_{F,H}(u,y) = \prod_{i\overset{H^c}{\sim} j} \kappa(u_i,u_j) \times  \vartheta_H(u,y),\qquad
  u\in [0,1]^k,\,y\in\IR^{\binom{k}{2}}.
}
and where empty products are understood to equal $1$.
\end{lemma}

\begin{proof}
We use induction over the number of edges in the graph $F$. If $F$ has no edges, the claim is clearly true, since $H=\emptyset\subseteq F$ is the only subgraph of $F$ without isolated vertices and in that case, $\prod_{i\overset{F}{\sim} j} y_{ij} = 1 = \rho_{F,\emptyset}(u,y)$.

Now, assume the assertion is true for all graphs with $e-1$ or fewer edges.
Let $F$ be a graph on $k$ vertices with $e$ edges. Fix an edge in $F$, say the edge between vertices $l$ and $m$, where $1\leq l<m\leq k$, and let $F_{lm}$ be the subgraph of $F$ obtained by removing that edge and any isolated vertex after the edge removal. Then 
\be{
\prod_{i\overset{F}{\sim} j} y_{ij} = y_{lm} \prod_{i\overset{F_{lm}}{\sim} j} y_{ij} 
 = (y_{lm}  -\kappa(u_l,u_m))  \prod_{i\overset{F_{lm}}{\sim} j} y_{ij} +  \kappa(u_l,u_m)  \prod_{i\overset{F_{lm}}{\sim} j} y_{ij} 
}
and by our assumption the decomposition holds for the subgraph $F_{lm}$, that is
\ba{
\prod_{i\overset{F_{lm}}{\sim} j}  y_{ij} & = \sum_{H\subseteq' F_{lm}} \prod_{ \{i,j\} \in E(F_{lm}) \setminus E(H)} \kappa(u_i,u_j) \times \prod_{i\overset{H}{\sim} j} \bclr{ y_{ij}-\kappa(u_i,u_j) }.
}
Thus we get
\besn{
\prod_{i\overset{F}{\sim} j}  y_{ij} 
= & \bclr{y_{lm}  -\kappa(u_l,u_m)}  \sum_{H\subseteq' F_{lm}} \prod_{  \{i,j\} \in E(F_{lm})\setminus E(H)} \kappa(u_i,u_j) \times\prod_{i\overset{H}{\sim} j} \bclr{y_{ij}-\kappa(u_i,u_j)}  \\
& +  \kappa(u_l,u_m) \sum_{H\subseteq' F_{lm}} \prod_{ \{i,j\} \in E(F_{lm})\setminus E(H)} \kappa(u_i,u_j) \times \prod_{i\overset{H}{\sim} j} \bclr{y_{ij}-\kappa(u_i,u_j)}  \\
= &\sum_{H\subseteq' F_{lm}}\prod_{\{i,j\} \in E(F_{lm})\setminus E(H)} \kappa(u_i,u_j) \times\prod_{i\overset{H}{\sim} j} \bclr{y_{ij}-\kappa(u_i,u_j)} \bclr{y_{lm}  -\kappa(u_l,u_m)}    \\
& +  \sum_{H\subseteq' F_{lm}} \prod_{ \{i,j\} \in E(F_{lm})\setminus E(H)} \kappa(u_i,u_j)\times \kappa(u_l,u_m) \times \prod_{i\overset{H}{\sim} j} \bclr{y_{ij}-\kappa(u_i,u_j)}  \\
= &  \sum_{\substack{H\subseteq' F:\\ \{l,m\}\in E(H)}}  \prod_{ \{i,j\} \in E(F)\setminus E(H)} \kappa(u_i,u_j) \times \prod_{i\overset{H}{\sim} j} \bclr{y_{ij}-\kappa(u_i,u_j)}  \\
& +  \sum_{\substack{H\subseteq' F:\\ \{l,m\}\not\in E(H)}}  \prod_{ \{i,j\} \in E(F)\setminus E(H)} \kappa(u_i,u_j) \times \prod_{i\overset{H}{\sim} j} \bclr{y_{ij}-\kappa(u_i,u_j)} \\
&=  \sum_{H\subseteq' F} \prod_{ \{i,j\} \in E(F)\setminus E(H)} \kappa(u_i,u_j) \times \prod_{i\overset{H}{\sim} j} \bclr{y_{ij}-\kappa(u_i,u_j)}.
}
Hence, the assertion is true for $F$, which completes the proof. 
\end{proof}

\begin{proof}[Proof of Lemma~\ref{lem1}] By Lemma~\ref{lem9},
\be{
  t^{\inj}_{F}(G_n) = \sum_{H\subseteq' F} s_{H}(U,Y),
}
where
\be{
  s_{H}(U,Y) = \frac{1}{(n)_k}\sum_{a\in\cA^n_k} \rho_{H}(U_a,Y_a)
}
(we drop dependence on $F$, since it is fixed).
Now, for $A\subset[k]$ (including the empty set), let
\be{
  M_A = \{\psi\in L_2([0,1]^k): \text{$\psi(u)$ depends on $(u_i)_{i\in A}$ only}\}
}
(in particular, $M_\emptyset$ consists of all constants)
and
\be{
  M_A^0 = \bclc{ \psi \in M_A:\text{$\IE\clc{\psi(U)\phi(U)}=0$ for all $B\subsetneq A$ and all $\phi\in M_B$}}.
}
From \cite[Lemma~11.17]{Janson1997}, it follows that, for any $\psi\in L_2([0,1]^k)$, there exists a unique orthogonal decomposition
\ben{\label{57}
  \psi(u) = \sum_{A\subseteq[k]} \psi_{A}(u), \qquad \psi_{A}\in M_A^0,\,A\subseteq[k].
}
Applying this to $\psi_H=\prod_{i\overset{H^c}{\sim} j} \kappa(u_i,u_j)=\sum_{A\subseteq [k]} \~\psi_{H,A}(u)$, we can decompose $s_{H}$ further into a sum of the form
\be{
  s_{H,A}(u,y) = \frac{1}{(n)_k} \sum_{a\in\cA^n_{k}} \~\psi_{H,A}(u_a)  \prod_{i\overset{H}{\sim} j} \bclr{y_{a_ia_j}-\kappa(u_{a_i},u_{a_j})}.
}
Let $l$ be the number of vertices in $H\cup A$; we can rewrite $s_{H,A}$ as $r_{H,A}$ 
where
\be{
r_{H,A}(u,y) =  \frac{1}{(n)_l} \sum_{a\in\cA^n_{l}} \psi_{H,A}\bclr{u_{a_{A_P}}}  \prod_{i\overset{H_P}{\sim} j} \bclr{y_{a_ia_j}-\kappa(u_{a_i},u_{a_j})},
}
with $\psi_{H,A}\clr{u}$, $u\in [0,1]^{\abs{A}}$, being the function obtained from $\~\psi_{H,A}\clr{u}$, $u\in [0,1]^k$, by a change of coordinates from the (now ordered) set $A$ to $(1,\dots,\abs{A})$. The claims about covariances and variances are straightforward to check.
\end{proof}

\begin{proof}[Proof of Lemma~\ref{lem2}]
Note that, for $\abs{A}\geq 2$, $M_A^0$ is the $L_2$-closure of the linear space spanned by
\be{
  \bbbclc{\prod_{i\in A} \psi_i(u_i): \psi_i\in L_2^\circ([0,1])}.
}
Hence, for any $\eps>0$, there are $N_{H,A}$ and  $\psi_{H,A,p,v}\in L_2^\circ([0,1])$, $v\in  [\abs{A}]$, $1\leq p\leq N$, such that
\be{
  \IE\bbbclr{ \psi_{H,A}(U) - \sum_{p=1}^{N_{H,A}} \prod_{i=1}^{\abs{A}}\psi_{H,A,p,i}(U_i)}^2\leq \eps,
}
and hence, for any $a\in\cA^n_k$,
\be{
  \IE\bbbbclr{\psi_{H,A}(U_{a_{A_P}}) \prod_{i\overset{H_P}{\sim} j} \bclr{Y_{ij}-\kappa(U_{i},U_{j})}-\sum_{p=1}^{N_{H,A}} \prod_{i=1}^{\abs{A}} \psi_{H,A,p,i}(U_{a_i})\prod_{i\overset{H_P}{\sim} j} \bclr{Y_{ij}-\kappa(U_{i},U_{j})}}^2
  \leq \eps,
}
since $\abs{\clr{Y_{a_ia_j}-\kappa(U_{a_i},U_{a_j})}}\leq 1$. 
With
\be{
  \~r_{H,A}(u,y) =  \frac{1}{(n)_l} \sum_{a\in\cA^n_{l}}\sum_{p=1}^{N_{H,A}} \prod_{i=1}^{\abs{A}} \psi_{H,A,p,i}(u_{a_i})\prod_{i\overset{H_P}{\sim} j} \bclr{y_{a_ia_j}-\kappa(u_{a_i},u_{a_j})}
}
we obtain
\bes{
  & \IE\bclr{r_{H,A}(U,Y)-\~r_{H,A}(U,Y)}^2\\
  &\enskip \leq \IE\bbbbclr{  \frac{1}{(n)_l}\sum_{a\in\cA^n_{l}}\bbbclr{\psi_{H,A}(U_a)-\sum_{p=1}^{N_{H,A}} \prod_{i=1}^{\abs{A}}\psi_{H,A,p,i}(U_{a_i})} \prod_{i\overset{H_P}{\sim} j} \bclr{Y_{a_ia_j}-\kappa(U_{a_i},U_{a_j})}}^2\\
  &\enskip \leq  \frac{l! \,\eps}{(n)_l} ,
}
where we have used that $\prod_{i=1}^{\abs{A}}\psi_{H,A,p,i}\in L_2^\circ\bclr{[0,1]^{\abs{A}}}$, so that all cross terms with $\abs{a\cap a'}>l$ vanish.
The final claim now follows from Lemma~\ref{lem10}.
\end{proof}

\begin{lemma}\label{lem10} Let $H$ be a graph on the vertex set $[l]$, and let $C_1,\dots,C_r$ denote the connected components of $H$. For each $1\leq i\leq r$, let $l_i$ be the size of $C_i$, let $C'_i$ be a graph on $[l_i]$ that is isomorphic to $C_i$, and let $\psi_i\in L_2^\circ([0,1]^{l_i})$. Let
\be{
  S_i(u,y) = \sum_{a\in\cA^n_{l_i}}\psi_i(u_a)\prod_{v\stackrel{C_i'}\sim w}\bclr{y_{a_va_w}-\kappa(u_{a_v},u_{a_w})}
}
Then 
\besn{\label{58}
  &\IE\bbbclr{ \sum_{a\in\cA^n_l}\prod_{i=1}^r\psi_i(U_{a_{V(C_i)}})\prod_{v\stackrel{C_i}\sim w}\bclr{y_{a_va_w}-\kappa(U_{a_v},U_{a_w})}-\prod_{i=1}^rS_i(U,Y)}^2\leq Cn^{l-1}
}
for all $u_v\in[0,1]$ and $y_{vw}\in\{0,1\}$, $1\leq v<w<n$. 
\end{lemma}

\begin{proof} When expanding the term $\prod_{i=1}^r \sum_{a\in \cA^n_{l_i}}$, consider two cases: either the different tuples of indices are all disjoint, or they overlap by at least one index. The first case easily yields the second expression in the difference \eqref{58}. For the second case, the size of the union of the indices can be at most $l-1$ which gives the order of the error in the approximation \eqref{58}.
\end{proof}

\section*{Supplementary Material}

The code to reproduce Table~\ref{tab1} can be found at \url{github.com/aroellin/csgc}.

\section*{Acknowledgements}

This project was supported by NUS Research Grant R-155-000-198-114. We thank Siva Athreya and Matas Sileikis for helpful discussions. We also thank the referee for helpful suggestions.

\setlength{\bibsep}{0.5ex}
\def\bibfont{\small}


\begin{thebibliography}{35}
\newcommand{\enquote}[1]{``#1''}
\expandafter\ifx\csname natexlab\endcsname\relax\def\natexlab#1{#1}\fi

\bibitem[\protect\citeauthoryear{Athreya, den Hollander, and
  R{\"o}llin}{Athreya et~al.}{2019}]{Athreya2019}
S.~Athreya, F.~den Hollander, and A.~R{\"o}llin (2019). Graphon-valued
  stochastic processes from population genetics.
\newblock \emph{arXiv preprint arXiv:1908.06241}.

\bibitem[\protect\citeauthoryear{Barbour, Karo{\'n}ski, and
  Ruci{\'n}ski}{Barbour et~al.}{1989}]{Barbour1989}
A.~D. Barbour, M.~Karo{\'n}ski, and A.~Ruci{\'n}ski (1989). A central limit
  theorem for decomposable random variables with applications to random graphs.
\newblock \emph{J.~Combin.~Theory~Ser.~B} {\bfseries 47}, 125--145.

\bibitem[\protect\citeauthoryear{Bentkus}{Bentkus}{2003}]{Bentkus2003}
V.~Bentkus (2003). On the dependence of the {B}erry-{E}sseen bound on
  dimension.
\newblock \emph{J. Statist. Plann. Inference} {\bfseries 113}, 385--402.

\bibitem[\protect\citeauthoryear{Biernacki, Celeux, and Govaert}{Biernacki et~al.}{2000}]{Biernacki2000}
C.~Biernacki, G.~Celeux, and G.~Govaert (2000). 
Assessing a mixture model for clustering with the integrated completed likelihood.
\newblock \emph{IEEE Trans. Pattern Anal. Mach. Intell.} \textbf{22}, 719–725.


\bibitem[\protect\citeauthoryear{Bollob{\'a}s and Riordan}{Bollob{\'a}s and
  Riordan}{2009}]{Bollobas2009}
B.~Bollob{\'a}s and O.~Riordan (2009). Metrics for sparse graphs.
\newblock In \emph{Surveys in Combinatorics 2009}, Cambridge University Press,
  vol. 365 of \emph{London Math. Soc. Lecture Note Ser.}, 211--287.

\bibitem[\protect\citeauthoryear{Borgs, Chayes, Cohn, and Holden}{Borgs
  et~al.}{2017}]{Borgs2017}
C.~Borgs, J.~T. Chayes, H.~Cohn, and N.~Holden (2017). Sparse exchangeable
  graphs and their limits via graphon processes.
\newblock \emph{J. Mach. Learn. Res.} {\bfseries 18}, 7740--7810.

\bibitem[\protect\citeauthoryear{Borgs, Chayes, Cohn, and Zhao}{Borgs
  et~al.}{2014{\natexlab{a}}}]{Borgs2014a}
C.~Borgs, J.~T. Chayes, H.~Cohn, and Y.~Zhao (2014{\natexlab{a}}). An {$L^p$}
  theory of sparse graph convergence {I}: limits, sparse random graph models,
  and power law distributions.
\newblock \emph{arXiv preprint arXiv:1401.2906}.

\bibitem[\protect\citeauthoryear{Borgs, Chayes, Cohn, and Zhao}{Borgs
  et~al.}{2014{\natexlab{b}}}]{Borgs2014b}
C.~Borgs, J.~T. Chayes, H.~Cohn, and Y.~Zhao (2014{\natexlab{b}}). An {$L^p$}
  theory of sparse graph convergence {II}: {LD} convergence, quotients, and
  right convergence.
\newblock \emph{arXiv preprint arXiv:1408.0744}.

\bibitem[\protect\citeauthoryear{Bubeck, Ding, Eldan, and R{\'a}cz}{Bubeck
  et~al.}{2016}]{Bubeck2016}
S.~Bubeck, J.~Ding, R.~Eldan, and M.~Z. R{\'a}cz (2016). Testing for
  high-dimensional geometry in random graphs.
\newblock \emph{Random Struct. Algorithms} {\bfseries 49}, 503--532.

\bibitem[\protect\citeauthoryear{Caron and Fox}{Caron and
  Fox}{2017}]{Caron2017}
F.~Caron and E.~B. Fox (2017). Sparse graphs using exchangeable random
  measures.
\newblock \emph{J. Roy. Statist. Soc. Ser. B} {\bfseries 79}, 1295--1366.

\bibitem[\protect\citeauthoryear{Chatterjee and Bhattacharya}{Chatterjee and
  Bhattacharya}{2021}]{Chatterjee2021}
Chatterjee, A., and Bhattacharya, B. B. (2021). Fluctuations of Subgraph Counts in Random Graphons. 
\newblock \emph{arXiv preprint arXiv:2104.07259}.

\bibitem[\protect\citeauthoryear{Chen and R{\"o}llin}{Chen and
  R{\"o}llin}{2010}]{Chen2010b}
L.~H.~Y. Chen and A.~R{\"o}llin (2010). Stein couplings for normal
  approximation.
\newblock \emph{arXiv preprint arXiv:1003.6039}.

\bibitem[\protect\citeauthoryear{Di~Nunno, {\O}ksendal, and Proske}{Di~Nunno
  et~al.}{2009}]{DiNunno2009}
G.~Di~Nunno, B.~K. {\O}ksendal, and F.~Proske (2009). \emph{Malliavin calculus
  for {L}{\'e}vy processes with applications to finance}.
\newblock vol.~2, Springer.

\bibitem[\protect\citeauthoryear{Diaconis and Janson}{Diaconis and
  Janson}{2008}]{Diaconis2008}
P.~Diaconis and S.~Janson (2008). Graph limits and exchangeable random graphs.
\newblock \emph{Rend. Mat. Appl. (7)} {\bfseries 28}, 33--61.

\bibitem[\protect\citeauthoryear{Fang and R{\"o}llin}{Fang and
  R{\"o}llin}{2015}]{Fang2012a}
X.~Fang and A.~R{\"o}llin (2015). Rates of convergence for multivariate normal
  approximation with applications to dense graphs and doubly indexed
  permutation statistics.
\newblock \emph{Bernoulli} {\bfseries 21}, 2157--2189.

\bibitem[\protect\citeauthoryear{Funke and Becker}{Funke and
  Becker}{2019}]{Funke2019}
T.~Funke and T.~Becker (2019). Stochastic block models: {A} comparison of
  variants and inference methods.
\newblock \emph{PLoS one} {\bfseries 14}.

\bibitem[\protect\citeauthoryear{Gan, R{\"o}llin, and Ross}{Gan
  et~al.}{2017}]{Gan2017}
H.~L. Gan, A.~R{\"o}llin, and N.~Ross (2017). Dirichlet approximation of
  equilibrium distributions in {C}annings models with mutation.
\newblock \emph{Adv. Appl. Probab.} {\bfseries 49}, 927--959.

\bibitem[\protect\citeauthoryear{Gao and Lafferty}{Gao and
  Lafferty}{2017{\natexlab{a}}}]{Gao2017a}
C.~Gao and J.~Lafferty (2017{\natexlab{a}}). Testing for global network
  structure using small subgraph statistics.
\newblock \emph{arXiv preprint arXiv:1710.00862}.

\bibitem[\protect\citeauthoryear{Gao and Lafferty}{Gao and
  Lafferty}{2017{\natexlab{b}}}]{Gao2017}
C.~Gao and J.~Lafferty (2017{\natexlab{b}}). Testing network structure using
  relations between small subgraph probabilities.
\newblock \emph{arXiv preprint arXiv:1704.06742}.

\bibitem[\protect\citeauthoryear{Hladk{\'y}, Pelekis and {\v{S}}ileikis}{Hladk{\'y} et~al.}{2019}]{Hladky2019}
J. Hladk{\'y}, C. Pelekis, and  M. {\v{S}}ileikis (2019). 
  A limit theorem for small cliques in inhomogeneous random graphs. 
  \emph{arXiv preprint arXiv:1903.10570}.

\bibitem[\protect\citeauthoryear{Hoppen, Kohayakawa, Moreira, and
  Sampaio}{Hoppen et~al.}{2011}]{Hoppen2011}
C.~Hoppen, Y.~Kohayakawa, C.~G. Moreira, and R.~M. Sampaio (2011). Limits of
  permutation sequences through permutation regularity.
\newblock \emph{arXiv preprint arXiv:1106.1663}.

\bibitem[\protect\citeauthoryear{Janson}{Janson}{1994}]{Janson1994}
S.~Janson (1994). Coupling and {P}oisson approximation.
\newblock \emph{Acta Appl. Math.} {\bfseries 34}, 7--15.

\bibitem[\protect\citeauthoryear{Janson}{Janson}{1997}]{Janson1997}
S.~Janson (1997). \emph{Gaussian {H}ilbert {S}paces}.
\newblock Cambridge University Press.

\bibitem[\protect\citeauthoryear{Janson and Nowicki}{Janson and
  Nowicki}{1991}]{Janson1991}
S.~Janson and K.~Nowicki (1991). The asymptotic distributions of generalized
  {$U$}-statistics with applications to random graphs.
\newblock \emph{Probab. Theory Related Fields} {\bfseries 90}, 341--375.

\bibitem[\protect\citeauthoryear{Krokowski, Reichenbachs, and
  Th{\"a}le}{Krokowski et~al.}{2017}]{Krokowski2015}
K.~Krokowski, A.~Reichenbachs, and C.~Th{\"a}le (2017). Discrete
  {M}alliavin-{S}tein method: {B}erry-{E}sseen bounds for random graphs and
  percolation. \newblock \emph{Ann. Probab.} {\bfseries 45}.


\bibitem[\protect\citeauthoryear{Lov{\'a}sz}{Lov{\'a}sz}{2012}]{Lovasz2012}
L.~Lov{\'a}sz (2012). \emph{Large Networks and Graph Limits}.
\newblock American Mathematical Society.

\bibitem[\protect\citeauthoryear{Lov{\'a}sz and Szegedy}{Lov{\'a}sz and
  Szegedy}{2006}]{Lovasz2006}
L.~Lov{\'a}sz and B.~Szegedy (2006). Limits of dense graph sequences.
\newblock \emph{J. Combin. Theory Ser. B} {\bfseries 96}, 933--957.

\bibitem[\protect\citeauthoryear{Lov{\'a}sz and Szegedy}{Lov{\'a}sz and
  Szegedy}{2011}]{Lovasz2011}
L.~Lov{\'a}sz and B.~Szegedy (2011). Finitely forcible graphons.
\newblock \emph{Journal of Combinatorial Theory, Series B} {\bfseries 101}, 269
  -- 301.


\bibitem[\protect\citeauthoryear{Maugis}{Maugis}{2020}]{Maugis2020}
P. A. Maugis (2020). Central limit theorems for local network statistics. \emph{arXiv preprint arXiv:2006.15738}.	


\bibitem[\protect\citeauthoryear{Meckes}{Meckes}{2009}]{Meckes2009}
E.~Meckes (2009). On {S}tein's method for multivariate normal approximation.
\newblock In \emph{High dimensional probability {V}: the {L}uminy volume},
  Beachwood, OH: Inst. Math. Statist., vol.~5 of \emph{Inst. Math. Stat.
  Collect.}, 153--178.

\bibitem[\protect\citeauthoryear{Nourdin and Peccati}{Nourdin and
  Peccati}{2012}]{Nourdin2012}
I.~Nourdin and G.~Peccati (2012). \emph{Normal approximation with {M}alliavin
  calculus: from {S}tein's method to universality}.
\newblock vol. 192 of \emph{Cambridge Tracts in Mathematics}, Cambridge:
  Cambridge University Press.

\bibitem[\protect\citeauthoryear{Nualart}{Nualart}{2006}]{Nualart2006}
D.~Nualart (2006). \emph{The {M}alliavin calculus and related topics}.
\newblock vol. 1995, Springer.

\bibitem[\protect\citeauthoryear{Nualart and Peccati}{Nualart and
  Peccati}{2005}]{Nualart2005}
D.~Nualart and G.~Peccati (2005). Central limit theorems for sequences of
  multiple stochastic integrals.
\newblock \emph{Ann. Probab.}

\bibitem[\protect\citeauthoryear{Ospina-Forero, Deane, and
  Reinert}{Ospina-Forero et~al.}{2019}]{OspinaForero2019}
L.~Ospina-Forero, C.~M. Deane, and G.~Reinert (2019). Assessment of model fit
  via network comparison methods based on subgraph counts.
\newblock \emph{J. Complex Netw.} {\bfseries 7}, 226--253.

\bibitem[\protect\citeauthoryear{Privault and Serafin}{Privault and
  Serafin}{2020}]{Privault2020}
N.~Privault and G.~Serafin (2020). Normal approximation for sums of weighted
  {$U$}-statistics --- application to {K}olmogorov bounds in random subgraph
  counting.
\newblock \emph{Bernoulli} {\bfseries 26}, 587--615.

\bibitem[\protect\citeauthoryear{R{\'a}th}{R{\'a}th}{2012}]{Rath2012a}
B.~R{\'a}th (2012). Time evolution of dense multigraph limits under
  edge-conservative preferential attachment dynamics.
\newblock \emph{Random Struct. Algorithms} {\bfseries 41}, 365--390.

\bibitem[\protect\citeauthoryear{R{\'a}th and Szak{\'a}cs}{R{\'a}th and
  Szak{\'a}cs}{2012}]{Rath2012}
B.~R{\'a}th and L.~Szak{\'a}cs (2012). Multigraph limit of the dense
  configuration model and the preferential attachment graph.
\newblock \emph{Acta Math. Hungar.} {\bfseries 136}, 196--221.

\bibitem[\protect\citeauthoryear{Reinert and R{\"o}llin}{Reinert and
  R{\"o}llin}{2010}]{Reinert2010}
G.~Reinert and A.~R{\"o}llin (2010). Random subgraph counts and
  {$U$}-statistics: {M}ultivariate normal approximation via exchangeable pairs
  and embedding.
\newblock \emph{J. Appl. Probab.} {\bfseries 47}, 378--393.

\bibitem[\protect\citeauthoryear{R{\"o}llin}{R{\"o}llin}{2017}]{Rollin2017}
A.~R{\"o}llin (2017). Kolmogorov bounds for the normal approximation of the
  number of triangles in the {E}rd{\H{o}}s-{R\'e}nyi random graph.
\newblock \emph{arXiv preprint arXiv:1704.00410}.

\bibitem[\protect\citeauthoryear{R{\"o}llin and Ross}{R{\"o}llin and
  Ross}{2015}]{Rollin2012a}
A.~R{\"o}llin and N.~Ross (2015). Local limit theorems via
  {L}andau-{K}olmogorov inequalities.
\newblock \emph{Bernoulli} {\bfseries 21}, 851--880.

\end{thebibliography}


\begin{appendix}
\section{Complete orthogonal decomposition of the triangle density for $2\times2$ block graphon}

Under the assumptions of Section~\ref{sec1-1}, we have
\be{
   t^{\inj}_\triangle(G_n) = R_{0.0} +  R_{0.5}+R_{1.0}+R_{1.5}+R_{2.0}+R_{2.5}
}
where
\ba{
  R_{0.0} & =(\alpha  \gamma+\beta  (1-\gamma) ) \bclr{\gamma  \left(\alpha ^2 \gamma +3 (1-\gamma) \delta ^2\right)
      -\alpha  \beta  \gamma  (1-\gamma)+\beta ^2 (1-\gamma)^2},\\
 R_{0.5} &=\frac{1}{n^{1/2}}\times \bclr{3 (\alpha  \gamma  (\alpha ^2 \gamma +(2-3 \gamma ) \delta ^2)-\beta ^3 (1-\gamma)^2+\beta  (3 \gamma ^2-4 \gamma +1) \delta ^2)}W, \\
  R_{1.0}&= \frac{1}{n-1}\times3\bclr{\alpha ^3 \gamma +\delta ^2 (\alpha(1-3\gamma) -\beta  (2-3 \gamma))+\beta ^3(1- \gamma)} (W^2 -\gamma(1-\gamma))\\    
  &\quad+\frac{1}{n^{1/2}(n-1)^{1/2}}\\
  &\qquad\times18^{1/2}\bbclr{
  (\alpha\gamma(\alpha -2\delta) +\beta ^2 (1-\gamma)-2 \beta  (1-\gamma) \delta +\delta ^2)V_{\edge,4}\\
  &\qquad\quad+(\alpha  \gamma +\beta  (1-\gamma)) (\alpha  \gamma -\beta  (1-\gamma)+ (1-2 \gamma)\delta )(V_{\edge,2}+V_{\edge,3})\\
  &\qquad\quad+\bclr{\gamma  (\alpha ^2 \gamma ^2-2 \alpha  (\gamma -1) \gamma  \delta -(\gamma -1) \delta ^2)-\beta ^2 (\gamma -1)^3+2 \beta  \gamma  (\gamma -1)^2 \delta} V_{\edge,1}},\\   
  R_{1.5}&=\frac{1}{n^{1/2}(n-1)}\times 3 \left(\alpha ^3 \gamma  (3 \gamma -2)-\alpha  \left(9 \gamma ^2-8 \gamma+1\right) \delta ^2-\beta ^3 \left(3 \gamma ^2-4 \gamma +1\right)+\beta  \left(9 \gamma ^2-10 \gamma +2\right) \delta ^2\right)W\\
  &\quad + \frac{n^{1/2}}{(n-1)(n-2)}\times\bclr{\alpha ^3+3 \delta ^2 (\beta -\alpha )-\beta ^3}\bclr{W^3 - n^{-1/2} \gamma(1-\gamma)  (1-2 \gamma )}\\
  &\quad + \frac{1}{n^{1/2}(n-1)^{1/2}(n-2)^{1/2}} \times 6^{1/2}\bclr{V_\triangle+V_{\twostar,1}+V_{\twostar,2}+V_{\twostar,3}}\\  
  &\quad +\frac{1}{(n-1)^{1/2}(n-2)}\\
  &\qquad\times 18^{1/2}\bbclr{(\alpha ^2 \gamma ^2+2 \alpha  \gamma  (1-\gamma) \delta -\beta ^2 (1-\gamma)^2-2 \beta  \gamma  (1-\gamma) \delta +(1-2 \gamma ) \delta ^2)V_{\edge,1}W \\
  &\qquad\quad + (\alpha ^2 \gamma +\alpha  (1 -2 \gamma)\delta+\beta ^2 (1-\gamma)-\beta  (1-2 \gamma) \delta -\delta ^2)(V_{\edge,2}+V_{\edge,3})W \\
  &\qquad\quad +(\alpha -\beta ) (\alpha +\beta -2 \delta )V_{\edge,4}W},\\
  R_{2.0}&=\frac{1}{n^{1/2}(n-1)^{1/2}(n-2)}\\
  &\qquad\times 18^{1/2}\bclr{2 (1-\gamma) \gamma  \left(\alpha ^2 (-\gamma )+\alpha  (2 \gamma -1) \delta +\beta ^2 (\gamma -1)+\beta  (\delta -2 \gamma  \delta )+\delta ^2\right) V_{\edge,1}\\
  &\qquad\quad+(\beta-\alpha) \left(2 \alpha  \gamma (1-\gamma) +2 \beta \gamma (1-\gamma) +(1-2 \gamma )^2 \delta\right)(V_{\edge,2}+V_{\edge,3})\\
  &\qquad\quad+2 \left(\alpha ^2 (\gamma -1)+\alpha  (\delta -2 \gamma  \delta )-\beta ^2 \gamma +\beta  (2 \gamma -1) \delta +\delta ^2\right) V_{\edge,4}}\\
  &\quad+ \frac{1}{(n-1)(n-2)} \times 3 (2 \gamma -1) (\alpha -\beta ) \left(\alpha ^2+\alpha  \beta +\beta ^2-3 \delta ^2\right)(W^2 -\gamma(1-\gamma)),\\ 
  R_{2.5}&= \frac{1}{n^{1/2}(n-1)(n-2)}\times 2 \left(6 \gamma ^2-6 \gamma +1\right) (\alpha -\beta ) \left(\alpha ^2+\alpha  \beta +\beta ^2-3 \delta ^2\right)W.
}

\end{appendix}

\end{document}